\documentclass[12pt]{amsart}
\usepackage{txfonts}      
\usepackage{amssymb}
\usepackage[mathscr]{eucal}
\usepackage{amsmath}
\usepackage{amscd}
\usepackage[dvips]{color}
\usepackage{multicol}
\usepackage[all]{xy}           
\usepackage{graphicx}
\usepackage{color}
\usepackage{colordvi}
\usepackage{xspace}
\usepackage[colorlinks,final,backref=page,hyperindex,hypertex]{hyperref}

\usepackage{tikz}

\begin{document}

\newcommand{\nc}{\newcommand}
\newcommand{\delete}[1]{}
\nc{\mfootnote}[1]{\footnote{#1}} 
\nc{\todo}[1]{\tred{To do:} #1}

\nc{\mlabel}[1]{\label{#1}}  
\nc{\mcite}[1]{\cite{#1}}  
\nc{\mref}[1]{\ref{#1}}  
\nc{\mbibitem}[1]{\bibitem{#1}} 

\delete{
\nc{\mlabel}[1]{\label{#1}  
{\hfill \hspace{1cm}{\bf{{\ }\hfill(#1)}}}}
\nc{\mcite}[1]{\cite{#1}{{\bf{{\ }(#1)}}}}  
\nc{\mref}[1]{\ref{#1}{{\bf{{\ }(#1)}}}}  
\nc{\mbibitem}[1]{\bibitem[\bf #1]{#1}} 
}

\newtheorem{theorem}{Theorem}[section]
\newtheorem{prop}[theorem]{Proposition}
\newtheorem{defn}[theorem]{Definition}
\newtheorem{lemma}[theorem]{Lemma}
\newtheorem{coro}[theorem]{Corollary}
\newtheorem{prop-def}[theorem]{Proposition-Definition}
\newtheorem{claim}{Claim}[section]
\newtheorem{remark}[theorem]{Remark}
\newtheorem{propprop}{Proposed Proposition}[section]
\newtheorem{conjecture}{Conjecture}
\newtheorem{exam}[theorem]{Example}
\newtheorem{assumption}{Assumption}
\newtheorem{condition}[theorem]{Assumption}
\newtheorem{question}[theorem]{Question}

\renewcommand{\labelenumi}{{\rm(\alph{enumi})}}
\renewcommand{\theenumi}{\alph{enumi}}

\nc{\tred}[1]{\textcolor{red}{#1}}
\nc{\tblue}[1]{\textcolor{blue}{#1}}
\nc{\tgreen}[1]{\textcolor{green}{#1}}
\nc{\tpurple}[1]{\textcolor{purple}{#1}}
\nc{\btred}[1]{\textcolor{red}{\bf #1}}
\nc{\btblue}[1]{\textcolor{blue}{\bf #1}}
\nc{\btgreen}[1]{\textcolor{green}{\bf #1}}
\nc{\btpurple}[1]{\textcolor{purple}{\bf #1}}


\nc{\vertset}{\Omega} 
\nc{\leafset}{\calx} 
\nc{\gensp}{V} 
\nc{\relsp}{R} 
\nc{\leafsp}{\mathcal{X}}    
\nc{\treesp}{\mathbb{T}} 
\nc{\genbas}{\mathcal{V}} 
\nc{\opd}{\mathcal{P}} 

\nc{\vin}{{\mathrm Vin}}    
\nc{\lin}{{\mathrm Lin}}    
\nc{\inv}{\mathrm{I}n}

\nc{\bvp}{V_P}     

\nc{\gop}{{\,\omega\,}}     
\nc{\gopb}{{\,\nu\,}}
\nc{\svec}[2]{{\textrm{\tiny{$\left(\begin{matrix}#1\\
#2\end{matrix}\right)$}}}}  
\nc{\ssvec}[2]{{\textrm{\tiny{$\left(\begin{matrix}#1\\
#2\end{matrix}\right)$}}}} 
\nc{\treeg}[5]{\vcenter{\xymatrix@M=1.5pt@R=1.5pt@C=0pt{#1 & & #2 & & & #3 \\ & #5 \ar@{-}[lu] \ar@{-}[ru] & & & & \\ & & #4 \ar@{-}[lu] \ar@{-}[rrruu] & & \\ & & & & & \\ & & \ar@{-}[uu] & & & }}}
\nc{\treed}[5]{\vcenter{\xymatrix@M=2pt@R=4pt@C=2pt{#1 & & & #2  & & #3 \\ & & & & #5 \ar@{-}[lu] \ar@{-}[ru] & \\ & & & #4 \ar@{-}[ru] \ar@{-}[llluu] & \\ & & & & & \\ & & & \ar@{-}[uu] & & }}}

\nc{\tsvec}[3]{{\textrm{\tiny{$\left(\begin{matrix}#1\\
#2\\#3\end{matrix}\right)$}}}}  

\nc{\stsvec}[3]{{\textrm{\tiny{$\left(\begin{matrix}#1\\
#2\\#3\end{matrix}\right)$}}}} 

\nc{\su}{\mathrm{Su}}
\nc{\bsu}{\mathrm{BSu}}
\nc{\tsu}{\mathrm{TSu}}
\nc{\TSu}{\mathrm{TSu}}
\nc{\eval}[1]{{#1}_{\big|D}}
\nc{\oto}{\leftrightarrow}

\nc{\du}{\mathrm{Du}}
\nc{\tdu}{\mathrm{Tri}}
\nc{\rep}{\dagger}

\nc{\oaset}{\mathbf{O}^{\rm alg}}
\nc{\omset}{\mathbf{O}^{\rm mod}}
\nc{\oamap}{\Phi^{\rm alg}}
\nc{\ommap}{\Phi^{\rm mod}}
\nc{\ioaset}{\mathbf{IO}^{\rm alg}}
\nc{\iomset}{\mathbf{IO}^{\rm mod}}
\nc{\ioamap}{\Psi^{\rm alg}}
\nc{\iommap}{\Psi^{\rm mod}}

\nc{\suc}{{successor}\xspace}
\nc{\Suc}{{Successor}\xspace}
\nc{\sucs}{{successors}\xspace}
\nc{\Sucs}{{Successors}\xspace}
\nc{\Tsuc}{{T-successor}\xspace}
\nc{\Tsucs}{{T-successors}\xspace}

\nc{\ddup}{{duplicator}\xspace}
\nc{\Ddup}{{Duplicator}\xspace}
\nc{\ddups}{{duplicators}\xspace}
\nc{\Ddups}{{Duplicators}\xspace}
\nc{\tdup}{{triplicator}\xspace}
\nc{\tdups}{{triduplicators}\xspace}
\nc{\Tdup}{{Triplicator}\xspace}
\nc{\Tdups}{{Triplicators}\xspace}

\nc{\perm}{{perm}}
\nc{\tperm}{{tri-perm}}

\nc{\Lsuc}{{L-successor}\xspace}
\nc{\Lsucs}{{L-successors}\xspace} \nc{\Rsuc}{{R-successor}\xspace}
\nc{\Rsucs}{{R-successors}\xspace}

\nc{\bia}{{$\mathcal{P}$-bimodule ${\bf k}$-algebra}\xspace}
\nc{\bias}{{$\mathcal{P}$-bimodule ${\bf k}$-algebras}\xspace}

\nc{\rmi}{{\mathrm{I}}}
\nc{\rmii}{{\mathrm{II}}}
\nc{\rmiii}{{\mathrm{III}}}

\nc{\pll}{\beta}
\nc{\plc}{\epsilon}

\nc{\ass}{{\mathit{Ass}}}
\nc{\comm}{{\mathit{Comm}}}
\nc{\comtrias}{{\mathit{ComTriass}}}
\nc{\TriLeib}{{\mathit{TriLeib}}}
\nc{\dend}{{\mathit{Dend}}}
\nc{\diass}{{\mathit{Diass}}}
\nc{\leib}{{\mathit{Leib}}}
\nc{\ldend}{{\mathit{LDend}}}
\nc{\lie}{{\mathit{Lie}}}
\nc{\lquad}{{\mathit{LQuad}}}
\nc{\octo}{{\mathit{Octo}}}
\nc{\Perm}{{\mathit{Perm}}}
\nc{\postlie}{{\mathit{PostLie}}}
\nc{\prelie}{{\mathit{preLie}}}
\nc{\Prelie}{{\mathit{PreLie}}}
\nc{\quado}{{\mathit{Quad}}}
\nc{\tridend}{{\mathit{TriDend}}}
\nc{\zinb}{{\mathit{Zinb}}}

 \nc{\adec}{\check{;}} \nc{\aop}{\alpha}
\nc{\dftimes}{\widetilde{\otimes}} \nc{\dfl}{\succ} \nc{\dfr}{\prec}
\nc{\dfc}{\circ} \nc{\dfb}{\bullet} \nc{\dft}{\star}
\nc{\dfcf}{{\mathbf k}} \nc{\apr}{\ast} \nc{\spr}{\cdot}
\nc{\twopr}{\circ} \nc{\tspr}{\star} \nc{\sempr}{\ast}
\nc{\disp}[1]{\displaystyle{#1}}
\nc{\bin}[2]{ (_{\stackrel{\scs{#1}}{\scs{#2}}})}  
\nc{\binc}[2]{ \left (\!\! \begin{array}{c} \scs{#1}\\
    \scs{#2} \end{array}\!\! \right )}  
\nc{\bincc}[2]{  \left ( {\scs{#1} \atop
    \vspace{-.5cm}\scs{#2}} \right )}  
\nc{\sarray}[2]{\begin{array}{c}#1 \vspace{.1cm}\\ \hline
    \vspace{-.35cm} \\ #2 \end{array}}
\nc{\bs}{\bar{S}} \nc{\dcup}{\stackrel{\bullet}{\cup}}
\nc{\dbigcup}{\stackrel{\bullet}{\bigcup}} \nc{\etree}{\big |}
\nc{\la}{\longrightarrow} \nc{\fe}{\'{e}} \nc{\rar}{\rightarrow}
\nc{\dar}{\downarrow} \nc{\dap}[1]{\downarrow
\rlap{$\scriptstyle{#1}$}} \nc{\uap}[1]{\uparrow
\rlap{$\scriptstyle{#1}$}} \nc{\defeq}{\stackrel{\rm def}{=}}
\nc{\dis}[1]{\displaystyle{#1}} \nc{\dotcup}{\,
\displaystyle{\bigcup^\bullet}\ } \nc{\sdotcup}{\tiny{
\displaystyle{\bigcup^\bullet}\ }} \nc{\hcm}{\ \hat{,}\ }
\nc{\hcirc}{\hat{\circ}} \nc{\hts}{\hat{\shpr}}
\nc{\lts}{\stackrel{\leftarrow}{\shpr}}
\nc{\rts}{\stackrel{\rightarrow}{\shpr}} \nc{\lleft}{[}
\nc{\lright}{]} \nc{\uni}[1]{\widetilde{#1}} \nc{\wor}[1]{\check{#1}}
\nc{\free}[1]{\bar{#1}} \nc{\den}[1]{\check{#1}} \nc{\lrpa}{\wr}
\nc{\curlyl}{\left \{ \begin{array}{c} {} \\ {} \end{array}
    \right .  \!\!\!\!\!\!\!}
\nc{\curlyr}{ \!\!\!\!\!\!\!
    \left . \begin{array}{c} {} \\ {} \end{array}
    \right \} }
\nc{\leaf}{\ell}       
\nc{\longmid}{\left | \begin{array}{c} {} \\ {} \end{array}
    \right . \!\!\!\!\!\!\!}
\nc{\ot}{\otimes} \nc{\sot}{{\scriptstyle{\ot}}}
\nc{\otm}{\overline{\ot}}
\nc{\ora}[1]{\stackrel{#1}{\rar}}
\nc{\ola}[1]{\stackrel{#1}{\la}}
\nc{\pltree}{\calt^\pl}
\nc{\epltree}{\calt^{\pl,\NC}}
\nc{\rbpltree}{\calt^r}
\nc{\scs}[1]{\scriptstyle{#1}} \nc{\mrm}[1]{{\rm #1}}
\nc{\dirlim}{\displaystyle{\lim_{\longrightarrow}}\,}
\nc{\invlim}{\displaystyle{\lim_{\longleftarrow}}\,}
\nc{\mvp}{\vspace{0.5cm}} \nc{\svp}{\vspace{2cm}}
\nc{\vp}{\vspace{8cm}} \nc{\proofbegin}{\noindent{\bf Proof: }}
\nc{\proofend}{$\blacksquare$ \vspace{0.5cm}}
\nc{\freerbpl}{{F^{\mathrm RBPL}}}
\nc{\sha}{{\mbox{\cyr X}}}  
\nc{\ncsha}{{\mbox{\cyr X}^{\mathrm NC}}} \nc{\ncshao}{{\mbox{\cyr
X}^{\mathrm NC,\,0}}}
\nc{\shpr}{\diamond}    
\nc{\shprm}{\overline{\diamond}}    
\nc{\shpro}{\diamond^0}    
\nc{\shprr}{\diamond^r}     
\nc{\shpra}{\overline{\diamond}^r}
\nc{\shpru}{\check{\diamond}} \nc{\catpr}{\diamond_l}
\nc{\rcatpr}{\diamond_r} \nc{\lapr}{\diamond_a}
\nc{\sqcupm}{\ot}
\nc{\lepr}{\diamond_e} \nc{\vep}{\varepsilon} \nc{\labs}{\mid\!}
\nc{\rabs}{\!\mid} \nc{\hsha}{\widehat{\sha}}
\nc{\lsha}{\stackrel{\leftarrow}{\sha}}
\nc{\rsha}{\stackrel{\rightarrow}{\sha}} \nc{\lc}{\lfloor}
\nc{\rc}{\rfloor}
\nc{\tpr}{\sqcup}
\nc{\nctpr}{\vee}
\nc{\plpr}{\star}
\nc{\rbplpr}{\bar{\plpr}}
\nc{\sqmon}[1]{\langle #1\rangle}
\nc{\forest}{\calf}
\nc{\altx}{\Lambda_X} \nc{\vecT}{\vec{T}} \nc{\onetree}{\bullet}
\nc{\Ao}{\check{A}}
\nc{\seta}{\underline{\Ao}}
\nc{\deltaa}{\overline{\delta}}
\nc{\trho}{\widetilde{\rho}}

\nc{\rpr}{\circ}
\nc{\dpr}{{\tiny\diamond}}
\nc{\rprpm}{{\rpr}}

\nc{\mmbox}[1]{\mbox{\ #1\ }} \nc{\ann}{\mrm{ann}}
\nc{\Aut}{\mrm{Aut}} \nc{\can}{\mrm{can}}
\nc{\twoalg}{{two-sided algebra}\xspace}
\nc{\colim}{\mrm{colim}}
\nc{\Cont}{\mrm{Cont}} \nc{\rchar}{\mrm{char}}
\nc{\cok}{\mrm{coker}} \nc{\dtf}{{R-{\rm tf}}} \nc{\dtor}{{R-{\rm
tor}}}
\renewcommand{\det}{\mrm{det}}
\nc{\depth}{{\mrm d}}
\nc{\Div}{{\mrm Div}} \nc{\End}{\mrm{End}} \nc{\Ext}{\mrm{Ext}}
\nc{\Fil}{\mrm{Fil}} \nc{\Frob}{\mrm{Frob}} \nc{\Gal}{\mrm{Gal}}
\nc{\GL}{\mrm{GL}} \nc{\Hom}{\mrm{Hom}} \nc{\hsr}{\mrm{H}}
\nc{\hpol}{\mrm{HP}} \nc{\id}{\mrm{id}} \nc{\im}{\mrm{im}}
\nc{\incl}{\mrm{incl}} \nc{\length}{\mrm{length}}
\nc{\LR}{\mrm{LR}} \nc{\mchar}{\rm char} \nc{\NC}{\mrm{NC}}
\nc{\mpart}{\mrm{part}} \nc{\pl}{\mrm{PL}}
\nc{\ql}{{\QQ_\ell}} \nc{\qp}{{\QQ_p}}
\nc{\rank}{\mrm{rank}} \nc{\rba}{\rm{RBA }} \nc{\rbas}{\rm{RBAs }}
\nc{\rbpl}{\mrm{RBPL}}
\nc{\rbw}{\rm{RBW }} \nc{\rbws}{\rm{RBWs }} \nc{\rcot}{\mrm{cot}}
\nc{\rest}{\rm{controlled}\xspace}
\nc{\rdef}{\mrm{def}} \nc{\rdiv}{{\rm div}} \nc{\rtf}{{\rm tf}}
\nc{\rtor}{{\rm tor}} \nc{\res}{\mrm{res}} \nc{\SL}{\mrm{SL}}
\nc{\Spec}{\mrm{Spec}} \nc{\tor}{\mrm{tor}} \nc{\Tr}{\mrm{Tr}}
\nc{\mtr}{\mrm{sk}}

\nc{\ab}{\mathbf{Ab}} \nc{\Alg}{\mathbf{Alg}}
\nc{\Algo}{\mathbf{Alg}^0} \nc{\Bax}{\mathbf{Bax}}
\nc{\Baxo}{\mathbf{Bax}^0} \nc{\RB}{\mathbf{RB}}
\nc{\DA}{\mathbf{DA}}
\nc{\RBo}{\mathbf{RB}^0} \nc{\BRB}{\mathbf{RB}}
\nc{\Dend}{\mathbf{DD}} \nc{\bfk}{{\bf k}} \nc{\bfone}{{\bf 1}}
\nc{\base}[1]{{a_{#1}}} \nc{\detail}{\marginpar{\bf More detail}
    \noindent{\bf Need more detail!}
    \svp}
\nc{\Diff}{\mathbf{Diff}} \nc{\gap}{\marginpar{\bf
Incomplete}\noindent{\bf Incomplete!!}
    \svp}
\nc{\FMod}{\mathbf{FMod}} \nc{\mset}{\mathbf{MSet}}
\nc{\rb}{\mathrm{RB}} \nc{\Int}{\mathbf{Int}}
\nc{\da}{\mathrm{DA}}
\nc{\Mon}{\mathbf{Mon}}
\nc{\remarks}{\noindent{\bf Remarks: }}
\nc{\OS}{\mathbf{OS}} 
\nc{\Rep}{\mathbf{Rep}}
\nc{\Rings}{\mathbf{Rings}} \nc{\Sets}{\mathbf{Sets}}
\nc{\DT}{\mathbf{DT}}

\nc{\BA}{{\mathbb A}} \nc{\CC}{{\mathbb C}} \nc{\DD}{{\mathbb D}}
\nc{\EE}{{\mathbb E}} \nc{\FF}{{\mathbb F}} \nc{\GG}{{\mathbb G}}
\nc{\HH}{{\mathbb H}} \nc{\LL}{{\mathbb L}} \nc{\NN}{{\mathbb N}}
\nc{\QQ}{{\mathbb Q}} \nc{\RR}{{\mathbb R}} \nc{\BS}{{\mathbb{S}}} \nc{\TT}{{\mathbb T}}
\nc{\VV}{{\mathbb V}} \nc{\ZZ}{{\mathbb Z}}


\nc{\calao}{{\mathcal A}} \nc{\cala}{{\mathcal A}}
\nc{\calc}{{\mathcal C}} \nc{\cald}{{\mathcal D}}
\nc{\cale}{{\mathcal E}} \nc{\calf}{{\mathcal F}}
\nc{\calfr}{{{\mathcal F}^{\,r}}} \nc{\calfo}{{\mathcal F}^0}
\nc{\calfro}{{\mathcal F}^{\,r,0}} \nc{\oF}{\overline{F}}
\nc{\calg}{{\mathcal G}} \nc{\calh}{{\mathcal H}}
\nc{\cali}{{\mathcal I}} \nc{\calj}{{\mathcal J}}
\nc{\call}{{\mathcal L}} \nc{\calm}{{\mathcal M}}
\nc{\caln}{{\mathcal N}} \nc{\calo}{{\mathcal O}}
\nc{\calp}{{\mathcal P}} \nc{\calq}{{\mathcal Q}} \nc{\calr}{{\mathcal R}}
\nc{\calt}{{\mathscr T}} \nc{\caltr}{{\mathcal T}^{\,r}}
\nc{\calu}{{\mathcal U}} \nc{\calv}{{\mathcal V}}
\nc{\calw}{{\mathcal W}} \nc{\calx}{{\mathcal X}}
\nc{\CA}{\mathcal{A}}

\nc{\fraka}{{\mathfrak a}} \nc{\frakB}{{\mathfrak B}}
\nc{\frakb}{{\mathfrak b}} \nc{\frakd}{{\mathfrak d}}
\nc{\oD}{\overline{D}}
\nc{\frakF}{{\mathfrak F}} \nc{\frakg}{{\mathfrak g}}
\nc{\frakm}{{\mathfrak m}} \nc{\frakM}{{\mathfrak M}}
\nc{\frakMo}{{\mathfrak M}^0} \nc{\frakp}{{\mathfrak p}}
\nc{\frakS}{{\mathfrak S}} \nc{\frakSo}{{\mathfrak S}^0}
\nc{\fraks}{{\mathfrak s}} \nc{\os}{\overline{\fraks}}
\nc{\frakT}{{\mathfrak T}}
\nc{\oT}{\overline{T}}
\nc{\frakX}{{\mathfrak X}} \nc{\frakXo}{{\mathfrak X}^0}
\nc{\frakx}{{\mathbf x}}
\nc{\frakTx}{\frakT}      
\nc{\frakTa}{\frakT^a}        
\nc{\frakTxo}{\frakTx^0}   
\nc{\caltao}{\calt^{a,0}}   
\nc{\ox}{\overline{\frakx}} \nc{\fraky}{{\mathfrak y}}
\nc{\frakz}{{\mathfrak z}} \nc{\oX}{\overline{X}}

\font\cyr=wncyr10

\nc{\redtext}[1]{\textcolor{red}{#1}}
\nc{\cm}[1]{\textcolor{blue}{Chengming: #1}}
\nc{\li}[1]{\textcolor{red}{Li: #1}}
\nc{\jun}[1]{\textcolor{blue}{Jun: #1}}
\nc{\xiang}[1]{\textcolor{purple}{Xiang: #1}}


\title{Replicating of binary operads, Koszul duality, Manin products and average operators}

\author{Jun Pei}
\address{Department of Mathematics, Lanzhou University, Lanzhou, Gansu 730000, China}
         \email{peitsun@163.com}

\author{Chengming Bai}
\address{Chern Institute of Mathematics \& LPMC, Nankai University, Tianjin 300071, China}
         \email{baicm@nankai.edu.cn}

\author{Li Guo}
\address{Department of Mathematics and Computer Science,
         Rutgers University,
         Newark, NJ 07102}
\email{liguo@rutgers.edu}

\author{Xiang Ni}
\address{Department of Mathematics, Caltech, Pasadena, CA 91125, USA} \email{xni@caltech.edu}

\date{\today}


\begin{abstract}
We consider the notions of the replicators, including the duplicator and triplicator, of a binary operad. As in the closely related notions of di-Var-algebra and tri-Var-algebra in~\mcite{GK2}, they provide a general operadic definition for the recent constructions of replicating the operations of algebraic structures. We show that taking replicators is in Koszul dual to taking successors in~\mcite{BBGN} for binary quadratic operads and is equivalent to taking the white product with certain operads such as $Perm$. We also relate the replicators to the actions of average operators.
\end{abstract}


\maketitle

\tableofcontents

\setcounter{section}{0}

\section{Introduction}
Motivated by the study of the periodicity in algebraic $K$-theory, J.-L. Loday~\mcite{Lo1} introduced the concept of a Leibniz algebra twenty years ago as a non-skew-symmetric generalization of the Lie algebra.
He then defined the diassociative algebra~\mcite{Lo2} as the enveloping algebra of the Leibniz algebra in analogue to the associative algebra as the enveloping algebra of the Lie algebra. The dendriform algebra was introduced as the Koszul dual of the diassociative algebra. These structures were studied systematically in the next few years in connection with operads~\mcite{Lo5}, homology~\mcite{Fr1,Fr2}, Hopf algebras~\mcite{AL,Hol,LR,Ron}, arithmetic~\mcite{Lo7}, combinatorics~\mcite{Fo,LR1}, quantum field theory~\mcite{Fo} and Rota-Baxter algebra~\mcite{Ag2}.

The diassociative and dendriform algebras extend the associative algebra in two directions. While the diassociative algebra ``doubles" the associative algebra in the sense that it has two associative operations with certain compatible conditions, the dendriform algebra ``splits" the associative algebra in the sense that it has two binary operations with relations between them so that the sum of the two operations is associative.

Into this century, more algebraic structures with multiple binary operations emerged, beginning with the triassociative algebra that ``triples" the associative algebra and the tridendriform algebra that gives a three way splitting of the associative algebra~\mcite{LR}. Since then, quite a few dendriform related structures, such as the quadri-algebra~\mcite{AL}, the ennea-algebra, the NS-algebra, the dendriform-Nijenhuis algebra, the octo-algebra~\mcite{Le1,Le2,Le3} and eventually a whole class of algebras~\mcite{Lo5,EG2} were introduced.
All these dendriform type structures have a
common property of ``splitting" the associativity into multiple pieces. Furthermore, analogues of the dendriform algebra, quadri-algebra and octo-algebra for the Lie algebra, commutative algebra, Jordan algebra, alternative algebra and Poisson algebra have been obtained~\mcite{Ag2,BLN,HNB,LNB,Lo4,NB}, such as the pre-Lie and Zinbiel algebras.
More recently, these constructions can be put into the framework of operad products (Manin black square and black dot products)~\mcite{EG,Lo3,Va}.

In~\mcite{BBGN}, the notions of ``successors" were introduced to give the precise meaning of two way and three way splitting of a binary operad and thus put the previous constructions in a uniform framework. This notion is also related to the Manin black products that had only been dealt with in special cases before, as indicated above. It is also shown to be related to the action of the Rota-Baxter operator, completing a long series of studies starting from the beginning of the century~\mcite{Ag2}.

In this paper, we take a similar approach to the other class of structures starting from the diassociative (resp. triassociative) algebra. That is, we seek to understand the phenomena of ``replicating" the operations in an operad.
After the completion of the paper, we realized that the closely related notions di-Var-algebra and tri-Var-algebra have been introduced in~\mcite{GK2} (see also~\mcite{Ko,KV}) by Kolesnikov and his coauthors. In fact their notions also apply to not necessarily binary operads~\mcite{KV}. We thank Kolesnikov for informing us to their studies. In this regards, the current paper provides an alternative and more detailed treatment of these notations for binary operads.

In Section~\mref{sec:conc} we set up a general framework to make precise the notion of ``replicating" any binary algebraic operad. This provides a general framework to study the previously well-known di-type (resp. tri-type) algebras which are analogues of the diassociative (resp. triassociative) algebra associated to the associative algebra, including the Leibniz algebra for the Lie algebra and the permutative algebra for the commutative algebra, as well as the recently defined pre-Lie dialgebra~\mcite{F}.
In general, it gives a ``rule" to construct new di-type (resp. tri-type) algebraic structures associated to any other binary operads. This notion is simpler in formulation but turns out to be equivalent to the notion of di-Var-algebra in~\mcite{GK2} for binary operads with nontrivial relations.

We show in Section~\mref{sec:mp} that taking the replicator of a binary quadratic operad is in Koszul dual with taking the successor of the dual operad. A direct application of this duality (Theorem~\mref{thm:dusu} and~\mref{thm:tdutsu}) is to explicitly compute the Koszul dual of the operads of existing algebras, for example the Koszul dual of the commutative tridendriform algebra of Loday~\mcite{Lo4}.
We also relate replicating to the Manin white product in the case of binary quadratic operads.
In fact taking the duplicator (resp. triplicator) of such an operad with nontrivial relations is isomorphic to taking the white product of the operad $\Perm$ (resp. $\comtrias$) with this operad, as in the case of taking di-Var-algebras and tri-Var-algebras\cite{GK2}. Thus showing the notations of duplicator and triplicator are equivalent to those of di-Var-algebras and tri-Var-algebras.

Finally, in Section~\mref{sec:rb}, we relate the replicating process to the action of average operators on binary quadratic operads.
Aguiar~\mcite{Ag2} showed that the action of the two-sided average operator on a commutative associative algebra (resp. associative algebra) gives a perm algebra (resp. associative dialgebra).
In \mcite{Uc}, Uchino extended the classical derived bracket construction to any algebra over a binary quadratic operad, showing that the derived bracket construction can be given by the Manin white product with the operad $\Perm$.

Thus there are relationship among the three operations applied to a binary operad $\calp$: taking its duplicator (resp. triplicator), taking its Manin white product with $\Perm$ (resp. $\comtrias$), when the operad is quadratic, and apply a di-average operator (resp. tri-average operator) to it, as summarized in the following diagram.
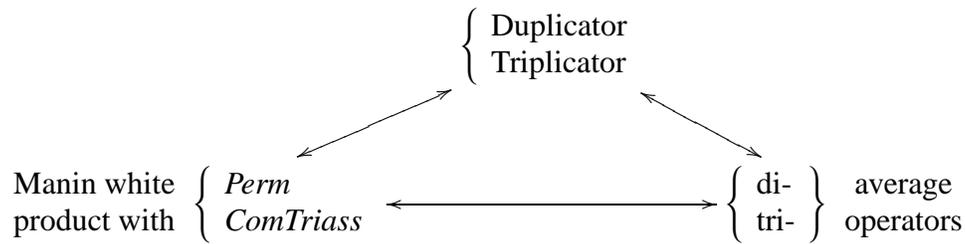
\begin{equation*}
 \xymatrix{
& \left\{ {\begin{array}{l} \text{Duplicator} \\ \text{Triplicator} \end{array} } \right .
 \ar@{<->}[ld] \ar@{<->}[rd] & \\
{\begin{array}{c}\text{Manin white}\\ \text{product with} \end{array} } \left\{ {\begin{array}{l}\Perm \\ \comtrias \end{array}} \right.  \ar@{<->}[rr]  && \left\{{\begin{array}{l} \text{di-} \\ \text{tri-}\end{array} } \right\} {\begin{array}{c} \text{average} \\ \text{operators} \end{array}}
}
\end{equation*}

Combining the replicators with the successors introduced in~\mcite{BBGN} allows us to put the splitting and replicating processes together, as exemplified in the following commutative diagram of operads. The arrows should be reversed on the level of categories.
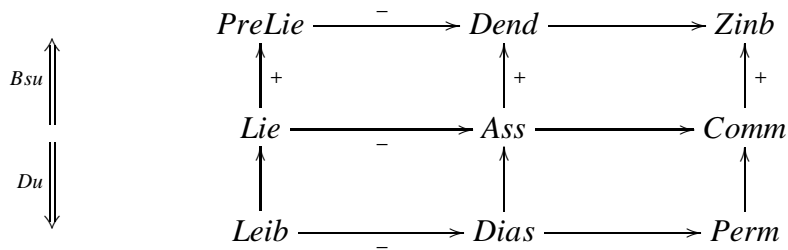
\begin{equation*}
\xymatrix{&&PreLie \ar@{->}^{-}[rr]& &Dend\ar[rr] & &Zinb \\
\ar@{=>}^{Bsu}[u]\ar@{=>}_{Du}[d]&&Lie \ar@{->}_{+}[u]\ar@{->}_{-}[rr] & & \ass\ar[rr] \ar@{->}_{+}[u] & & Comm\ar@{->}_{+}[u]\\
&&Leib \ar[u]\ar@{->}_{-}[rr]& & Dias \ar[rr] \ar[u]& & Perm\ar[u]
}
\end{equation*}
Here the vertical arrows in the upper half of the diagram are addition of the two operations given in ~\cite[Proposition~2.31.(a)]{BBGN} while those in the lower half of the diagram are given in Proposition~\mref{prop:quotient} (\ref{it:quotienta}). The horizontal arrows in the left half of the diagram are anti-symmetrization of the binary operations while those in the right half of the diagram are induced by the identity maps on the binary operations.
In the diagram, the Koszul dual of an operad is the reflection across the center. A similar commutative diagram holds for the trisuccessors and triplicators.

\section{The replicators of a binary  operad}
\mlabel{sec:conc}
In this section, we first introduce the concepts of the replicators, namely the duplicator and triplicator, of a labeled planar binary tree, following the framework in~\mcite{BBGN} and in close resemblance with the concepts of the di-Var-algebra tri-Var-algebra in~\mcite{GK2}. These concepts are then applied to define similar concepts for a nonsymmetric operad and a (symmetric) operad. A list of examples is provided, followed by a study of the relationship among an operad, its duplicator and its triplicator.

\subsection{The replicators of a planar binary tree}
We first recall notions on operads represented by trees. For more details see~\mcite{BBGN,LV}.
\subsubsection{Labeled trees}
\begin{defn}
{\rm
\begin{enumerate}
\item
Let $\calt$ denote the set of planar binary reduced rooted trees together with the trivial tree $\vcenter{\xymatrix@M=4pt@R=8pt@C=4pt{\ar@{-}[d]\\ \\}}$. If $t\in \calt$ has $n$ leaves, we call $t$ an {\bf $n$-tree}. The trivial tree $\vcenter{\xymatrix@M=4pt@R=8pt@C=4pt{\ar@{-}[d]\\ \\}}$ has one leaf.
\item
 Let $\vertset$ be a set. By a {\bf decorated tree} we mean a tree $t$ of $\calt$ together with a decoration on the vertices of $t$ by elements of $\vertset$ and a decoration on the leaves of $t$ by distinct positive integers. Let $t(\vertset)$ denote the set of decorated trees for $t$ and denote
\begin{equation*}
\calt(\vertset):=\coprod_{ t \in \calt} t(\vertset).
\end{equation*}
If $\tau\in t(\vertset)$ for an $n$-tree $t$, we call $\tau$ a {\bf labeled $n$-tree}.
\item
For $\tau\in \calt(\vertset)$, we let $\vin(\tau)$ (resp. $\lin(\tau)$) denote the set (resp. ordered set) of labels of the vertices (resp. leaves) of $\tau$.
\item
Let $\tau\in \calt(\vertset)$ with $|\lin(\tau)|>1$ be a labeled tree from $t\in \calt$. Then $t$ can be written uniquely as the grafting $t_\ell\vee t_r$ of $t_\ell$ and $t_r$. Correspondingly, let $\tau=\tau_\ell\vee_{\gop} \tau_r$ denote the unique decomposition of $\tau$ as a grafting of $\tau_\ell$ and $\tau_r$ in $\calt(\vertset)$ along $\gop\in \vertset$.
\end{enumerate}}
\end{defn}

Let $V$ be a vector space, regarded as an arity graded vector space concentrated in arity 2: $V=V_2$. Recall~\cite[Section 5.8.5]{LV} that the free nonsymmetric operad $\mathcal{T}_{\hspace*{-0.1cm}ns}(\gensp)$ on $V$ is given by the vector space
$$ \mathcal{T}_{\hspace*{-0.1cm}ns}(\gensp) := \displaystyle{ \bigoplus_{t \in \calt } } \ t[\gensp] \ , $$
where $t[\gensp]$ is the treewise tensor module associated to $t$, explicitly given by
$$  t[\gensp] := \displaystyle{\bigotimes_{v \in \vin(t)}} \gensp_{|\inv(v)|} \ . $$
Here $|\inv(v)|$ denotes the number of incoming edges of $v$. A basis $\genbas$ of $\gensp$ induces a basis $t(\genbas)$ of $t[\gensp]$ and a basis $\calt(\genbas)$ of $\mathcal{T}_{\hspace*{-0.1cm}ns}(\gensp)$. Consequently any element of $t[\gensp]$ can be represented as a linear combination of elements in $t(\genbas)$.

\subsubsection{Duplicators}

\begin{defn} \mlabel{defn:vector}
{\rm
Let $\gensp$ be a vector space with a basis $\genbas$.
\begin{enumerate}
\item
Define a vector space
\begin{equation}
\du(\gensp)=\gensp \ot (\bfk \dashv \oplus ~\bfk \vdash) \ ,
\mlabel{eq:tsp}
\end{equation}
where we denote $(\gop \otimes \dashv)$ (resp. $(\gop \otimes \vdash)$) by $\svec{\gop}{\dashv}$ $\Big($resp. $\svec{\gop}{\vdash}$$\Big)$ for $\gop\in \genbas$. Then $\displaystyle \bigcup_{\gop \in \genbas }\left\{\svec{\gop}{\dashv},\svec{\gop}{\vdash}\right\}$ is a basis of $\du(\gensp)$.
\item
For a labeled $n$-tree $\tau$ in $\calt(\genbas)$, define  a subset $\du(\tau)$ of $\mathcal{T}_{\hspace*{-0.1cm}ns}(\du(\genbas))$ by
\begin{itemize}
 \item[$\bullet$] $\du(\,\vcenter{\xymatrix@M=0pt@R=8pt@C=4pt{\ar@{-}[d]\\ \\}}\,)=\big\{\vcenter{\xymatrix@M=4pt@R=8pt@C=4pt{\ar@{-}[d]\\ \\}}\big\}$,
\item[$\bullet$] when $n\geq 2$, $\du(\tau)$ is obtained by replacing each decoration $\gop\in \vin(\tau)$ by $$\svec{\gop}{\rep}:=\left\{\svec{\gop}{\dashv}, \svec{\gop}{\vdash}\right\} \ . $$
\end{itemize}
\end{enumerate}
}
\end{defn}
Thus $\du(\tau)$ is a set of labeled trees.

\begin{defn}
{\rm
Let $\gensp$ be a vector space with a basis $\genbas$. Let $\tau$ be a labeled $n$-tree in $\calt(\genbas)$.
The {\bf duplicator} $\du_{x}(\tau) $ of $\tau$ with respect to a leaf $x\in \lin(\tau)$ is the subset of $ \mathcal{T}_{\hspace*{-0.1cm}ns}(\du(\genbas))$ defined by induction on $|\lin(\tau)|$ as follows:
\begin{enumerate}
\item[$\bullet$] $\du_x(\vcenter{\xymatrix@M=4pt@R=8pt@C=4pt{\ar@{-}[d]\\ \\}})=\left\{\,\vcenter{\xymatrix@M=0pt@R=8pt@C=4pt{\ar@{-}[d]\\ \\}}\,\right\}$ \ ;
\item[$\bullet$] assume that $\du_x(\tau)$ have been defined for $\tau$ with $|\lin(\tau)|\leq k$ for a $k\geq 1$. Then, for a labeled $(k+1)$-tree $\tau\in \calt(\genbas)$ with decomposition $\tau=\tau_\ell \vee_{\gop} \tau_r$, we define
\begin{equation*}
\du_{x}(\tau)=\du_{x}(\tau_\ell\vee_{\gop} \tau_r) = \left \{\begin{array}{ll}
\du_{x}(\tau_\ell)\vee_{\ssvec{\gop}{\dashv}} \du(\tau_{r}), & x\in \lin(\tau_\ell), \\
\du(\tau_{\ell})\vee_{\ssvec{\gop}{\vdash}} \du_{x}(\tau_r),& x\in \lin(\tau_r). \end{array} \right .
\end{equation*}
\end{enumerate}
For labeled $n$-trees $\tau_i, 1\leq i\leq r,$ with the same set of leaf decorations and $c_i\in \bfk, 1\leq i\leq r$, define
\begin{equation}
\du_x\left(\sum_{i=1}^r c_i \tau_i\right): = \sum_{i=1}^r c_i\du_x(\tau_i).
\end{equation}
Here and in the rest of the paper we use the notation \begin{equation}
\sum_{i=1}^r c_i W_i:=
\left\{ \sum_{i=1}^r c_i w_i\,\Big|\, w_i\in W_i, 1\leq i\leq r\right\},
\mlabel{eq:ssum}
\end{equation}
for nonempty subsets $W_i, 1\leq i\leq r$, of a $\bfk$-module.
}
\end{defn}
The next explicit description of the duplicator follows from an induction on $|\lin(\tau)|$.
\begin{prop}\label{reppath}
Let $V$ be a vector space with a basis $\genbas$, $\tau$ be in $\calt(\genbas)$ and $x$ be in $\lin(\tau)$. The duplicator $\du_x(\tau)$ is obtained by relabeling a vertex $\gop$ of $\vin(\tau)$ by
$$\left\{\begin{array}{ll}
\svec{\gop}{\dashv}\,, &  \text{the path from the root of } \tau \text{ to } x \text{ turns left at } \gop;\\
\svec{\gop}{\vdash}\,, & \text{the path from the root of } \tau \text{ to } x \text{ turns right at } \gop;\\
\svec{\gop}{\rep}:=\left\{\svec{\gop}{\dashv}, \svec{\gop}{\vdash}\right\}\,, & \text{the path from the root of } \tau \text{\ to\ } x \text{ does not pass } \gop.
\end{array}
\right.$$
\end{prop}

\begin{exam}
${\rm Du}_{x_2} \left( \vcenter{\xymatrix@M=2pt@R=4pt@C=4pt{
x_1 \ar@{-}[dr] & & x_2 & & x_3 \ar@{-}[dr] & & x_4 \ar@{-}[dl]\\
 & \gop_1 \ar@{-}[ur] & & & & \gop_3 \ar@{-}[ddll] & \\
 & & & & & & \\
 & & & \gop_2 \ar@{-}[uull] \ar@{-}[d] & & & \\
 & & & & & & \\
 & & & \ar@{-}[uu] & & &\\
}}\right) = $
$$\vcenter{\xymatrix@M=2pt@R=4pt@C=4pt{
x_1 \ar@{-}[dr] & & x_2 & & x_3 \ar@{-}[dr] & & x_4 \ar@{-}[dl]\\
 & \ssvec{\gop_1}{\vdash} \ar[ur] & & & & \ssvec{\gop_3}{\rep} \ar@{-}[ddll] & \\
 & & & & & & \\
 & & & \ssvec{\gop_2}{\dashv} \ar[uull] & & & \\
 & & & & & & \\
 & & & \ar[uu] & & \\
}}= \left\{\vcenter{\xymatrix@M=2pt@R=4pt@C=4pt{
x_1 \ar@{-}[dr] & & x_2 & & x_3 \ar@{-}[dr] & & x_4 \ar@{-}[dl]\\
 & \ssvec{\gop_1}{\vdash} \ar@{-}[ur] & & & & \ssvec{\gop_3}{\dashv} \ar@{-}[ddll] & \\
 & & & & & & \\
 & & & \ssvec{\gop_2}{\dashv} \ar@{-}[uull] & & & \\
 & & & & & & \\
 & & & \ar@{-}[uu] & & \\
}} ,  \vcenter{\xymatrix@M=2pt@R=4pt@C=4pt{
x_1 \ar@{-}[dr] & & x_2 & & x_3 \ar@{-}[dr] & & x_4 \ar@{-}[dl]\\
 & \ssvec{\gop_1}{\vdash} \ar@{-}[ur] & & & & \ssvec{\gop_3}{\vdash} \ar@{-}[ddll] & \\
 & & & & & & \\
 & & & \ssvec{\gop_2}{\dashv} \ar@{-}[uull] & & & \\
 & & & & & & \\
 & & & \ar@{-}[uu] & & \\
}} \right\}$$
\end{exam}

\subsubsection{Triplicators}
\begin{defn}
{\rm Let $\gensp$ be a vector space with a basis $\genbas$.
\begin{enumerate}
\item
Define a vector space
\begin{equation}
\tdu(\gensp)=\gensp\ot (\bfk \dashv \oplus \ \bfk \vdash \oplus \ \bfk \perp) \ ,
\mlabel{eq:ttsp}
\end{equation}
where we denote $(\gop \otimes \dashv)$ (resp. $(\gop \otimes \vdash)$, resp. $(\gop \otimes \perp)$) by $\svec{\gop}{\dashv}$ $\Big($resp. $\svec{\gop}{\vdash}$, resp. $\svec{\gop}{\perp}$$\Big)$ for $\gop\in \genbas$. Then $\displaystyle \bigcup_{\gop \in \genbas }\left\{\svec{\gop}{\dashv},\svec{\gop}{\vdash}, \svec{\gop}{\perp}\right\}$ is a basis of $\tdu(\gensp)$.
\item
Let $\tau$ be a labeled $n$-tree in $\calt(\genbas)$ and let $J$ be a subset of $\lin(\tau)$.
The {\bf triplicator} $\tdu_{J}(\tau)$ of $\tau$ with respect to $J$ is a subset of $ \mathcal{T}_{\hspace*{-0.1cm}ns}(\tdu(\genbas))$ defined by induction on $|\lin(\tau)|$ as follows:
\begin{enumerate}
\item[$\bullet$] $\tdu_J(\vcenter{\xymatrix@M=4pt@R=8pt@C=4pt{\ar@{-}[d]\\ \\}})=\big\{\,\vcenter{\xymatrix@M=0pt@R=8pt@C=4pt{\ar@{-}[d]\\ \\}}\,\big\}$ \ ;
\item[$\bullet$] assume that $\tdu_J(\tau)$ have been defined for $\tau$ with $|\lin(\tau)|\leq k$ for a $k\geq 1$. Then, for a labeled $(k+1)$-tree $\tau\in \calt(\genbas)$ with decomposition $\tau=\tau_\ell \vee_{\gop} \tau_r$, we define
$$ \tdu_{J}(\tau)=\tdu_J(\tau_\ell\vee_\gop \tau_r)= \tdu_{J\cap\lin(\tau_\ell)} \vee_{\ssvec{\gop}{(\tau,J)}} \tdu_{J\cap\lin(\tau_r)},$$
where
$$ (\tau,J)=\left\{ \begin{array}{ll}
    \dashv, & J\cap \lin(\tau_\ell)\neq \varnothing, J\cap \lin(\tau_r)=\varnothing, \text{ that is, } J\subseteq \lin(\tau_r),\\
    \vdash, & J\cap \lin(\tau_\ell)= \varnothing, J\cap \lin(\tau_r)\neq\varnothing, \text{ that is, } J\subseteq \lin(\tau_\ell),\\
    \rep:=\{\dashv,\vdash,\perp\}, & J\cap \lin(\tau_\ell)= \varnothing, J\cap \lin(\tau_r)=\varnothing, \text{ that is, } J=\varnothing, \\
    \perp, & J\cap \lin(\tau_\ell)\neq \varnothing, J\cap \lin(\tau_r)\neq\varnothing, \text{ that is,  none of the above}.
\end{array} \right .
    $$
Equivalently,
\begin{equation*}
\tdu_{J}(\tau)= \left \{\begin{array}{ll}
\tdu_{J}(\tau_\ell) \vee_{\ssvec{\gop}{\dashv}} \tdu_\varnothing(\tau_{r}), & J\subseteq \lin(\tau_\ell), \\
\tdu_\varnothing(\tau_{\ell})\vee_{\ssvec{\gop}{\vdash}} \tdu_{J}(\tau_r),& J\subseteq \lin(\tau_r),\\
\tdu_\varnothing (\tau_\ell)\vee_{\ssvec{\gop}{\rep}} \tdu_\varnothing (\tau_r), & J=\varnothing, \\
\tdu_{J\cap\lin(\tau_\ell)}(\tau_\ell) \vee_{\ssvec{\gop}{\perp}} \tdu_{J\cap\lin(\tau_r)}(\tau_r), & \text{ otherwise}. \end{array} \right .
\end{equation*}
\end{enumerate}
\end{enumerate}
}
\end{defn}

We have the following explicit description of the triplicator that follows from an induction on $|\lin(\tau)|$.

\begin{prop}\mlabel{treppath}
Let $V$ be a vector space with a basis $\calv$, let $\tau$ be in $\calt(\genbas)$ and let $J$ be a nonempty subset of $\lin(\tau)$. The triplicator $\tdu_J(\tau)$ is obtained by relabeling each vertex $\gop$ of $\vin(\tau)$ by the following rules:
\begin{enumerate}
\item Suppose $\gop$ is on the paths from the root of $\tau$ to some (possibly multiple) $x$ in $J$. Then
    \begin{enumerate}
    \item[(i)] replace $\gop$ by $\svec{\gop}{\dashv}$ if all of such paths turn left at $\gop$;
    \item[(ii)] replace $\gop$ by $\svec{\gop}{\vdash}$ if all of such paths turn right at $\gop$;
    \item[(iii)] replace $\gop$ by $\svec{\gop}{\perp}$ if some of such paths turn left at $\gop$ and some of such paths turn right at $\gop$.
    \end{enumerate}
\item Suppose $\gop$ is not on the path from the root of $\tau$ to any $x\in J$. Then replace $\gop$ by $\svec{\gop}{\rep}:=\left\{\svec{\gop}{\dashv} , \svec{\gop}{\vdash}, \svec{\gop}{\perp}\right\}$;
\end{enumerate}
\end{prop}

\begin{exam}
$\tdu_{\{1,2\}} \left( \vcenter{\xymatrix@M=2pt@R=4pt@C=4pt{
1 \ar@{-}[dr] & & 2 & & 3 \ar@{-}[dr] & & 4 \ar@{-}[dl]\\
 & \gop_1 \ar@{-}[ur] & & & & \gop_3 \ar@{-}[ddll] & \\
 & & & & & & \\
 & & & \gop_2 \ar@{-}[uull] \ar@{-}[d] & & & \\
 & & & & & & \\
 & & & \ar@{-}[uu] & & &\\
}}\right) = \vcenter{\xymatrix@M=2pt@R=4pt@C=4pt{
1  & & 2 & & 3 & & 4 \ar@{-}[dl]\\
 & \ssvec{\gop_1}{\perp} \ar[ul] \ar[ur]& & & & \ssvec{\gop_3}{\rep} \ar@{-}[ul] & \\
 & & & & & & \\
 & & & \ssvec{\gop_2}{\dashv} \ar[uull] \ar@{-}[uurr] & & & \\
 & & & & & & \\
 & & & \ar[uu] & & \\
}}$
$$ = \left\{\vcenter{\xymatrix@M=2pt@R=4pt@C=4pt{
1 \ar@{-}[dr] & & 2 & & 3 \ar@{-}[dr] & & 4 \ar@{-}[dl]\\
 & \ssvec{\gop_1}{\perp} \ar@{-}[ur] & & & & \ssvec{\gop_3}{\dashv} \ar@{-}[ddll] & \\
 & & & & & & \\
 & & & \ssvec{\gop_2}{\dashv} \ar@{-}[uull] & & & \\
 & & & & & & \\
 & & & \ar@{-}[uu] & & \\
}},
\vcenter{\xymatrix@M=2pt@R=4pt@C=4pt{
1 \ar@{-}[dr] & & 2 & & 3 \ar@{-}[dr] & & 4 \ar@{-}[dl]\\
 & \ssvec{\gop_1}{\perp} \ar@{-}[ur] & & & & \ssvec{\gop_3}{\vdash} \ar@{-}[ddll] & \\
 & & & & & & \\
 & & & \ssvec{\gop_2}{\dashv} \ar@{-}[uull] & & & \\
 & & & & & & \\
 & & & \ar@{-}[uu] & & \\
}} ,
\vcenter{\xymatrix@M=2pt@R=4pt@C=4pt{
1 \ar@{-}[dr] & & 2 & & 3 \ar@{-}[dr] & & 4 \ar@{-}[dl]\\
 & \ssvec{\gop_1}{\perp} \ar@{-}[ur] & & & & \ssvec{\gop_3}{\perp} \ar@{-}[ddll] & \\
 & & & & & & \\
 & & & \ssvec{\gop_2}{\dashv} \ar@{-}[uull] & & & \\
 & & & & & & \\
 & & & \ar@{-}[uu] & & \\
}}\right\}$$

\end{exam}

\subsection{The replicators of a binary nonsymmetric operad}\label{repnsopd}

\begin{defn}
{\rm
Let $\gensp$ be a vector space with a basis $\calv$.
\begin{enumerate}
\item
An element
$$r:=\sum_{i=1}^r c_{i}\tau_{i}, \quad c_{i}\in\bfk, \tau_i\in \calt(\genbas),$$
in $\mathcal{T}_{\hspace*{-0.1cm}ns}(\gensp)$ is called {\bf homogeneous} if $\lin(\tau_i)$ are the same for $1\leq i\leq r$. Then denote $\lin(r)=\lin(\tau_i)$ for any $1\leq i\leq r$.
\item
A collection of elements
$$r_s:=\sum_{i=1}^r c_{s,i}\tau_{s,i}, \quad c_{s,i}\in\bfk, \tau_{s,i}\in \calt(\genbas), 1\leq s\leq k, k\geq 1, $$
in $\mathcal{T}_{\hspace*{-0.1cm}ns}(\gensp)$ is called {\bf locally homogenous} if each element $r_s$, $1\leq s\leq k$, is homogeneous.
\end{enumerate}
}
\end{defn}

\begin{defn}
{\rm Let $\opd=\mathcal{T}_{\hspace*{-0.1cm}ns}(V)/(R)$ be a binary nonsymmetric operad where $V$ is a vector space with a basis $\genbas$ regarded as an arity graded vector space concentrated in arity two:  $V=V_2$ and $R$ is a set consisting of locally homogeneous elements:
\begin{equation*}
r_s=\sum_i c_{s,i}\tau_{s,i}\ \in \mathcal{T}_{\hspace*{-0.1cm}ns}(\gensp)\; , \; \ c_{s,i}\in\bfk, \ \tau_{s,i}\in \calt(\genbas), \ 1\leq s\leq k.
\end{equation*}
\begin{enumerate}
\item
The {\bf duplicator} of $\opd$ is defined to be the binary nonsymmetric operad $$
\du(\opd):=\mathcal{T}_{\hspace*{-0.1cm}ns}(\du(V) )/(\du(R)).
$$
Here $\du(V)=V\ot(\bfk \dashv \oplus\, \bfk\vdash)$ is regarded as an arity graded vector space concentrated in arity two and
\begin{equation*}
\du(R):= \bigcup_{s=1}^k\left(\bigcup_{x \in Lin(t_{s})}\du_x(r_{s})\right),
\ \text{ where } \du_x(r_{s}):= \sum_{i}c_{s,i}\du_{x}(\tau_{s,i})),
\end{equation*}
with the notation in Eq.~(\mref{eq:ssum}).
\item The {\bf triplicator} of $\opd$ is defined to be the  binary nonsymmetric operad
    $$
    \tdu(\opd):= \mathcal{T}_{\hspace*{-0.1cm}ns}(\tdu(V))/(\tdu(R)).$$
Here $\tdu(V)=V\ot(\bfk \dashv \oplus\, \bfk\vdash \oplus\, \bfk\perp)$ is regarded as an arity graded vector space concentrated in arity two and
\begin{equation*}
\tdu(R):= \bigcup_{s=1}^k\left(\bigcup_{\varnothing\neq J \subseteq Lin(r_s)}
\tdu_{J}(r_{s})\right),
\text{ where }
\tdu_{J}(r_{s}) := \sum_i c_{s,i}\tdu_{J}(\tau_{s,i}).
\end{equation*}
\end{enumerate}
}
\end{defn}

\begin{prop}\label{indgenbas}
The duplicator (resp. triplicator) of a binary nonsymmetric operad $\opd=\mathcal{T}_{\hspace*{-0.1cm}ns}(V)/(R)$ does not depend on the choice of a basis $\genbas$ of $V$.
\end{prop}
\begin{proof}
It is straightforward to check from the linearity of the duplicator
(resp. triplicator) and from the treewise tensor module structure on $\mathcal{T}_{\hspace*{-0.1cm}ns}(\gensp)$.
\end{proof}

We give some examples of duplicators and triplicators of nonsymmetric operads.

\begin{exam}
{\rm Let $\ass$ be the nonsymmetric operad of the associative algebra with product $\cdot$. Using the abbreviations $\dashv:=\svec{\cdot}{\dashv}$ and
$\vdash:=\svec{\cdot}{\vdash}$, we have
\begin{eqnarray*}
&\du_{y}((x \cdot y) \cdot z - x \cdot (y \cdot z)) = \{(x \vdash y)\dashv z-x \vdash (y \dashv z)\},&\\
&\du_{x}((x \cdot y) \cdot z - x \cdot (y \cdot z)) = \{(x \dashv y)\dashv z-x \dashv (y \dashv z), (x \dashv y)\dashv z-x \dashv (y \vdash z)\},&\\
&\du_{z}((x \cdot y) \cdot z - x \cdot (y \cdot z)) = \{(x \dashv y)\vdash z-x \vdash(y \vdash z), (x \vdash y)\vdash z-x \vdash(y \vdash z)\},&
\end{eqnarray*}
giving the five relations of the
{\bf diassociative algebra} of
Loday~\mcite{Lo2}. Therefore the duplicator of $\ass$ is $\diass$.
}
\end{exam}

\begin{exam}
{\rm A similar computation shows that the triplicator of $\ass$ is the operad $\mathit{Trias}$ of the {\bf triassociative algebra} of Loday and Ronco~\mcite{LR}. For example,
\begin{eqnarray*}
&\tdu_{\{x\}}((xy)z- x(yz)) = \{(x \dashv y) \dashv z - x \dashv (y \dashv z), (x \dashv y) \dashv z - x \dashv (y \vdash z),(x \dashv y) \dashv z - x \dashv (y \perp z)\},&\\
&\tdu_{\{x,y\}}((xy)z- x(yz)) = \{(x \perp y) \dashv z - x \perp (y \dashv z)\},&\\
&\tdu_{\{x,y,z\}}((xy)z- x(yz)) = \{(x \perp y) \perp z - x \perp (y \perp z)\}.&
\end{eqnarray*}
}
\end{exam}

\subsection{The replicators of a binary operad}

When $\gensp=\gensp(2)$ is an $\BS$-module concentrated in arity two with a linear basis $\genbas$.
%
For any finite set $\mathcal{X}$ of cardinal $n$, define the coinvariant space
$$ \gensp(\mathcal{X}):= \left( \displaystyle{\bigoplus_{f: \underline{n}\rightarrow \mathcal{X}}} \gensp(n) \right)_{\BS_{n}} \ ,$$
where the sum is over all the bijections from $\underline{n}:=\{ 1, \ldots , n \}$ to $\mathcal{X}$ and where the symmetric group acts diagonally.

Let $\mathbb{T}$ denote the set of isomorphism classes of reduced binary trees~\cite[Appendix C]{LV}. For $\textsf{t}\in \mathbb{T}$, define the treewise tensor $\BS$-module associated to $\textsf{t}$, explicitly given by
$$  \textsf{t}[\gensp] := \displaystyle{\bigotimes_{v \in \vin(\textsf{t})} } \gensp(\inv(v)) \ , $$
see~\cite[Section 5.5.1]{LV}.
Then the free operad $\mathcal{T}(\gensp)$ on an $\BS$-module $\gensp=\gensp(2)$ is given by the $\BS$-module
$$ \mathcal{T}(\gensp) := \displaystyle{ \bigoplus_{\textsf{t} \in \mathbb{T} } } \ \textsf{t}[\gensp] \,. $$

Each tree $\textsf{t}$ in $\mathbb{T}$ can be represented by a planar tree $t$ in $\calt$ by choosing a total order on the set of inputs of each vertex of $\textsf{t}$. Further, $t[\gensp]\cong \mathsf{t}[\gensp]$~\cite[Section 2.8]{Ho}. Fixing such a choice $t$ for each $\textsf{t}\in \mathbb{T}$ gives a subset $\mathfrak{R}\subseteq \calt$ with a bijection $ \mathbb{T} \cong \mathfrak{R}$.
Then we have
$$\mathcal{T}(\gensp)\cong \displaystyle{ \bigoplus_{t \in \mathfrak{R} } } \ t[\gensp] \ , $$
allowing us to use the notations in Section~\ref{repnsopd}.

\begin{defn}\mlabel{rule}
{\rm Let $\opd=\mathcal{T}(\gensp)/ (\relsp)$ be a binary operad where the $\BS$-module $\gensp$ is concentrated in arity $2$: $\gensp=\gensp(2)$ with an $\BS_2$-basis $\genbas$ and the space of relations is generated, as an $\BS$-module, by a set $R$ of locally homogeneous elements
\begin{equation}
r_s:=\sum_i c_{s,i}\tau_{s,i}, \ c_{s,i}\in\bfk, \tau_{s,i} \in \bigcup_{t\in \mathfrak{R}} t(\genbas), \ 1\leq s\leq k.
\mlabel{eq:pres}
\end{equation}
\begin{enumerate}
\item The {\bf duplicator} of $\opd$ is defined to be the binary operad
$$\du(\opd)=\mathcal{T}(\du(\gensp))/ (\du(\relsp))$$
where
the $\BS_2$-action on $\du(\gensp)=V\ot \left (\bfk\dashv \oplus\, \bfk \vdash\right)$ is given by
\begin{equation*}
\svec{\gop}{\dashv}^{(12)}:=\svec{\gop^{(12)}}{\vdash} \; ,\quad
\svec{\gop}{\vdash}^{(12)}:=\svec{\gop^{(12)}}{\dashv}\; , \;  \gop\in \gensp,
\end{equation*}
and the space of relations is generated, as an $\BS$-module, by
\begin{equation}
\du(R):= \bigcup_{s=1}^k\left(\bigcup_{x \in Lin(r_{s})}\du_x(r_{s})\right)
\text{ with }
\du_x(r_{s}):= \sum_ic_{s,i}\du_{x}(\tau_{s,i}).
\mlabel{reprelation}
\end{equation}
\item The {\bf triplicator} of $\opd$ is defined to be the binary operad
    $$\tdu(\opd)=\mathcal{T}(\tdu(\gensp))/ (\tdu(\relsp))
    $$
where the $\BS_2$-action on $\tdu(\gensp)=V\ot \left (\bfk\dashv \oplus\, \bfk \vdash\oplus\, \bfk \perp\right)$ is given by
\begin{equation*}
\svec{\gop}{\dashv}^{(12)}:=\svec{\gop^{(12)}}{\vdash} \; ,\quad \svec{\gop}{\vdash}^{(12)}:=\svec{\gop^{(12)}}{\dashv}\; , \quad
\svec{\gop}{\perp}^{(12)}:=\svec{\gop^{(12)}}{\perp} \; ,\;  \gop\in \gensp,
\end{equation*}
and the space of relations is generated, as an $\BS$-module, by
\begin{equation*}
\tdu(R):= \bigcup_{s=1}^k\left(\bigcup_{\varnothing\neq J \subseteq Lin(r_s)}
\tdu_{J}(r_{s})\right)
\text{ with }
\tdu_{J}(r_{s}) :=\sum_i c_{s,i}\tdu_{J}(\tau_{s,i}).
\end{equation*}
\end{enumerate}
}
\end{defn}
See\mcite{GK2} for the closely related notions of the di-Var-algebra and tri-Var-algebra, and~\mcite{KV} for these notions for not necessarily binary operads.
For later reference, we also recall the definitions of bisuccessors~\cite{BBGN}.

\begin{defn}
{\rm The {\bf bisuccessor}~\cite{BBGN} of a binary operad $\opd=\mathcal{T}(\gensp)/(\relsp)$ is defined to be the binary operad $\su(\opd)=\mathcal{T}( \widetilde{\gensp}  )/ (\su(\relsp))$ where
the $\BS_2$-action on $\widetilde{\gensp}$ is given by
\begin{equation}
\svec{\gop}{\prec}^{(12)}:=\svec{\gop^{(12)}}{\succ} \; ,\quad
\svec{\gop}{\succ}^{(12)}:=\svec{\gop^{(12)}}{\prec}\; , \;  \gop\in \gensp,
\notag 
\end{equation}
and the space of relations is generated, as an $\BS$-module, by
\begin{equation}
\su(R):= \left\lbrace  \su_x(r_{s}):=\sum_ic_{s,i}\su_{x}(t_{s,i})\;\Big| \; x\in \lin(r_s), \; 1\leq s\leq
k  \right\rbrace . \mlabel{eq:sup'}
\end{equation}
Here for $\tau \in  \mathcal{T}(V)$ and a leaf $x \in \lin(\tau)$, $\su_{x}(\tau)$ is defined by relabeling a vertex $\omega$ of $\vin(\tau)$ by
$$\left\{\begin{array}{ll}
\svec{\omega}{\prec}\,, &  \text{the path from the root of } \tau \text{ to } x \text{ turns left at } \omega;\\
\svec{\omega}{\succ}\,, & \text{the path from the root of } \tau \text{ to } x \text{ turns right at } \omega;\\
\svec{\omega}{\star}\,,& \text{$\omega$ is not on the path from the root of } \tau \text{\ to\ } x,
\end{array}
\right.$$
where $\svec{\omega}{\star}:=\left\{\svec{\omega}{\prec}+ \svec{\omega}{\succ}\right\}$.
}
\mlabel{de:bsu}
\end{defn}

There is a similar notion of a trisuccessor splitting an operation into three pieces ~\cite{BBGN}.

With an argument similar to the proof of Proposition 2.20 in \mcite{BBGN}, we see that the duplicator and triplicator of a binary algebraic operad $\opd=\mathcal{T}(\gensp)/ (R)$ depends neither on the linear basis $\genbas$ of $\gensp$ nor on the set $\mathfrak{R}$.

\subsection{Examples of duplicators and triplicators}
\mlabel{ss:exam}

We give some examples of duplicators and triplicators of binary operads.

Let $V$ be an $\BS$-module concentrated in arity two.
Then we have
$$\mathcal{T}(V)(3)=(V\otimes_{\BS_2}(V\otimes\bfk\oplus\bfk\otimes V))\otimes_{\BS_2}\bfk[\BS_3],$$
which can be identify with 3 copies of $V\otimes V$, denoted by $V\circ_{{\rm I}}V,V\circ_{{\rm II}}V$ and $V\circ_{{\rm III}}V$, following the convention in~\mcite{Va}.
Then, as an abelian group, $\mathcal{T}(V)(3)$ is generated by elements of the form
\begin{equation}
\gop \circ_{\rmi} \gopb (\oto (x\gopb y)\gop z), \gop\circ_{\rmii} \gopb (\oto (y\gopb z)\gop x), \gop\circ_{\rmiii} \gopb (\oto (z\gopb x)\gop y), \forall\gop, \gopb\in \gensp.
\mlabel{eq:type}
\end{equation}

For an operad where the space of generators $V$ is equal to $\bfk[\BS_2]=\mu.\bfk\oplus\mu'.\bfk$ with $\mu.(12)=\mu'$, we will adopt the convention in~\cite[p. 129]{Va} and denote the 12 elements of $\mathcal{T}(V)(3)$ by $v_i, 1\leq i\leq 12,$ in the following table.
\begin{center}
\begin{tabular}{|c|c|c|c|c|c|}
  \hline
  $v_1$ &  $\mu\circ_{{\rm I}}\mu\leftrightarrow(xy)z$ & $v_{5}$ & $\mu\circ_{{\rm III}}\mu\leftrightarrow(zx)y$ & $v_{9}$ & $\mu\circ_{{\rm II}}\mu\leftrightarrow(yz)x$ \\ \hline
 $v_{2}$ & $\mu'\circ_{{\rm II}}\mu\leftrightarrow x(yz)$ & $v_{6}$ & $\mu'\circ_{{\rm I}}\mu\leftrightarrow z(xy)$ & $v_{10}$ & $\mu'\circ_{{\rm III}}\mu\leftrightarrow y(zx)$ \\ \hline
  $v_{3}$ & $\mu'\circ_{{\rm II}}\mu'\leftrightarrow x(zy)$ & $v_{7}$ & $\mu'\circ_{{\rm I}}\mu'\leftrightarrow z(yx)$ & $v_{11}$ & $\mu'\circ_{{\rm III}}\mu'\leftrightarrow y(xz)$ \\ \hline
  $v_{4}$ & $\mu\circ_{{\rm III}}\mu'\leftrightarrow(xz)y$ & $v_{8}$ & $\mu\circ_{{\rm II}}\mu'\leftrightarrow(zy)x$ & $v_{12}$ & $\mu\circ_{{\rm I}}\mu'\leftrightarrow(yx)z$ \\
  \hline
\end{tabular}
\end{center}

\subsubsection{Examples of duplicators}
Recall that a {\bf (left) Leibniz algebra}~\mcite{Lo2} is defined by a bilinear operation $\{,\}$ and a relation $$\{x,\{y,z\}\} = \{\{x,y\},z\} + \{y,\{x,z\}\}.$$
\begin{prop}\label{leibniz}
The operad $\leib$ of the Leibniz algebra is the duplicator of $\lie$, the operad of the Lie algebra.
\mlabel{pp:leib}
\end{prop}

\begin{proof}
Let $\mu$ denote the operation of the operad $\lie$. The space of relations of $\lie$ is generated as an $\BS_{3}$-module by
\begin{equation}
v_1+v_5+v_9=\mu\circ_{{\rm I}}\mu+\mu\circ_{{\rm II}}\mu+\mu\circ_{{\rm III}}\mu=(x\mu y)\mu z+(z\mu x)\mu y+(y\mu z)\mu x.
\mlabel{eq:lie}
\end{equation}
Use the abbreviations $\dashv:=\svec{\mu}{\dashv}$ and
$\vdash:=\svec{\mu}{\vdash}$. Then from $\svec{\mu}{\dashv}^{(12)}=\svec{\mu^{(12)}}{\vdash}=-\svec{\mu}{\vdash}$, we have
$\dashv^{(12)}=-\vdash$.
Then we have
\small{\begin{eqnarray*}
\du_z(v_1+v_5+v_9)&=&\{(x \vdash y)\vdash z + (y \vdash z) \dashv x + (z \dashv x)\dashv y,(x \dashv y)\vdash z + (y \vdash z) \dashv x + (z \dashv x)\dashv y\}\\
&=&\{\underline{(x \vdash y) \vdash z - x \vdash (y \vdash z) + y \vdash (x \vdash z)}, y \vdash (x \vdash z) - (y \vdash x) \vdash z - x \vdash (y \vdash z)\},
\end{eqnarray*}
}
with similar computations for $\du_x$ and $\du_y$.
Replacing the operation $\vdash$ by $\{,\}$, we see that the underlined relation is precisely the relation of the Leibniz algebra while the other relations are obtained from this relation by a permutation of the variables. Therefore $\du(\lie)=\leib$.
\end{proof}

Also recall that a {\bf (left) permutative
algebra}~\mcite{Ch} (also called commutative diassociative algebra) is defined by one bilinear operation $\cdot$ and
the relations
$$ x \cdot (y \cdot z)= (x \cdot y) \cdot z = (y \cdot x) \cdot z.$$

\begin{prop}
The operad $\Perm$ of the permutative algebra is the
duplicator of $\comm$, the operad of the commutative associative algebra.
\mlabel{pp:perm}
\end{prop}
\begin{proof}
Let $\gop$ denote the operation of the operad $\comm$. Setting $\dashv:=\svec{\gop}{\dashv}$ and
$\vdash:=\svec{\gop}{\vdash}$, then from $\svec{\gop}{\vdash}^{(12)} =\svec{\gop^{(12)}}{\vdash}=\svec{\gop}{\vdash}$ we have
$\dashv^{(12)}=\vdash$. The space of relations of $\comm$ is generated as an $\BS_{3}$-module by $$v_1 - v_9=\gop \circ_{{\rm I}}\gop - \gop \circ_{{\rm II}}\gop = (x \gop y) \gop z - (y \gop z) \gop x.$$
Then we have
\begin{eqnarray*}
\du_z(v_1 - v_9)&=&\{(x \dashv y) \vdash z - (y \vdash z) \dashv x, (x \vdash y) \vdash z - (y \vdash z) \dashv x\}\\
&=&\{\underline{(y \vdash x) \vdash z - x \vdash (y \vdash z),(x \vdash y) \vdash z - x \vdash (y \vdash z)}\},
\end{eqnarray*}
with similar computations for $\du_x$ and $\du_y$. Replacing the operation $\vdash$ by $\cdot$ and following the same proof as in Proposition~\mref{pp:leib}, we get $\du(\comm)=\Perm$.
\end{proof}

A {\bf (left) Poisson algebra} is
defined to be a $\bfk$-vectors space with two bilinear operations $\{,\}$ and $\circ$ such that
$\{,\}$ is the Lie bracket and $\circ$ is the product of commutative
associative algebra, and they are compatible in the sense that
$$\{x,y\circ z\}=\{x,y\}\circ z+y\circ\{x,z\}.$$
A {\bf dual (left) pre-Poisson algebra}~\mcite{Ag2} is defined to be a $\bfk$-vector space with two bilinear operations $\{\,,\}$ and $\circ$ such that
$\{\,,\}$ is a Leibniz bracket and $\circ$ is a product of permutative algebra, and they are compatible in the sense that
$$
\{x,y\circ z\}=\{x,y\} \circ z + y \circ \{x, z\},\
\{x \circ y, z\} = x \circ \{y,z\} + y \circ \{x,z\},\
\{x,y\} \circ z = -\{y,x\} \circ z.
$$

By a similar argument as in Proposition~\mref{pp:leib}, we obtain
\begin{prop}
The duplicator of {\it Pois}, the operad of the Poisson algebra, is
{\it DualPrePois}, the operad of the dual pre-Poisson algebra.
\mlabel{pp:prepois}
\end{prop}

We next consider the duplicator of the the operad $\prelie$ of {\bf (left) pre-Lie} algebra (also called left-symmetric algebra). A pre-Lie algebra is defined by a bilinear operation $\{~,~\}$ that satisfies
$$
R_{\prelie}:= \{\{x,y\},z\} - \{x,\{y,z\}\} - \{\{y,x\},z\} + \{y,\{x,z\}\} =0.
$$

By Definition~\mref{rule} and the abbreviations $\dashv:=\svec{\gop}{\dashv}, \vdash:=\svec{\gop}{\vdash}$, we have
\begin{eqnarray*}
\du(R_{\prelie}) &=& \left\{ \underline{x \dashv (y \dashv z) - x \dashv (y \vdash z)},~y \dashv (x \dashv z) - y \dashv (x \vdash z),\right.\\
&&~\underline{(x \vdash y) \vdash z - (x \dashv y) \vdash z},~(y \vdash x) \vdash z - (y \dashv x) \vdash z,\\
&&~\underline{x \dashv (y \dashv z) - (x \dashv y) \dashv z -y \vdash(x \dashv z) + (y \vdash x) \dashv z}, \\
&&~x \vdash (y \dashv z) - (x \vdash y) \dashv z - y \dashv(x \dashv z) + (y \dashv x) \dashv z, \\
&&\left.~\underline{x \vdash (y \vdash z) - (x \vdash y) \vdash z  - y \vdash(x \vdash z) + (y \vdash x) \vdash z}\right\}
\end{eqnarray*}
These underline relations coincide with the axioms of preLie dialgebra (left-symmetric dialgebra) defined in~\mcite{F}, and the other relations are obtained from this relation by a permutation of the variables. Then we have

\begin{prop}
The duplicator of $\prelie$, the operad of the pre-Lie algebra,
 is {\it DipreLie}, the operad of the pre-Lie dialgebra.
\end{prop}

\subsubsection{Examples of triplicators}
We similarly have the following examples of triplicators of operads.

A {\bf commutative trialgebra}~\mcite{LR} is a vector space $A$ equipped with a product $\star$ and a commutative product $\bullet$ satisfying the following equations:
$$
(x \star y) \star z = \star (y \star z),
x \star (y \star z) =x \star (y \bullet z),
x \bullet (y \star z) =(x \bullet y) \star z,
(x \bullet y) \bullet z =x \bullet (y \bullet z).
$$

\begin{prop}
The operad {\it ComTrias} of the commutative trialgebra is the triplicator of  {\it Comm}.
\end{prop}
\begin{proof}
Let $\gop$ be the operation of the operad $\comm$. Set $\dashv:=\svec{\gop}{\dashv}$,
$\vdash:=\svec{\gop}{\vdash}$ and $\perp:=\svec{\gop}{\perp}$. Since $\svec{\gop}{\dashv}^{(12)}=\svec{\gop^{(12)}}{\vdash}=\svec{\gop}{\vdash}$
and $\svec{\gop}{\perp}^{(12)}=\svec{\gop^{(12)}}{\perp}=\svec{\gop}{\perp}$, we have
$\dashv^{(12)}=\vdash$ and $\perp^{(12)}=\;\perp$.The space of relations of $\comm$ is generated as an $\BS_{3}$-module by $$v_1 - v_9=\gop \circ_{{\rm I}}\gop - \gop \circ_{{\rm II}}\gop = (x \gop y) \gop z - (y \gop z) \gop x.$$
Then we have, for example,
\begin{eqnarray*}
\tdu_x(v_1 - v_9)&=& \{ (x \dashv y) \dashv z - (y \dashv z) \vdash x,  (x \dashv y) \dashv z - (y \vdash z) \vdash x, (x \dashv y) \dashv z -(y \perp z) \vdash x\}\\
&=&\{\underline{(x \dashv y) \dashv z - x \dashv (y \dashv z),  (x \dashv y) \dashv z - x \dashv (z \dashv y), (x \dashv y) \dashv z-  x \dashv (y \perp z)}\};\\
\tdu_{\{x,y\}}(v_1 - v_9)&=&\{\underline{(x \perp y) \dashv z - (y \dashv z) \perp x}\};\\
\tdu_{\{x,y,z\}}(v_1 - v_9)&=& \{(x \perp y) \perp z - (y \perp z) \perp x\} = \{\underline{(x \perp y) \perp z - x \perp (y \perp z)}\}.
\end{eqnarray*}
Replacing the operation $\dashv$ by $\star$ and $\perp$ by $\bullet$, we see that the underlined relations are equivalent to the relations of the commutative trialgebra.  The other relations can be obtained from these relations by a permutation of the variables and the commutativity of $\perp$.
Thus we get $\tdu(\comm)={\it ComTrias}$.
\end{proof}

We next consider the triplicator of $\lie$.
Let $\mu$ be the operation of the operad $\lie$. Set $\dashv:=\svec{\mu}{\dashv}$,
$\vdash:=\svec{\mu}{\vdash}$ and $\perp:=\svec{\mu}{\perp}$. Since $\svec{\mu}{\dashv}^{(12)}=\svec{\mu^{(12)}}{\vdash}=-\svec{\mu}{\vdash}$
and $\svec{\mu}{\perp}^{(12)}=\svec{\mu^{(12)}}{\perp}=-\svec{\mu}{\perp}$, we have
$\dashv^{(12)}=-\vdash$ and $\perp^{(12)}=-\;\perp$. The space of relations of $\lie$ is generated as an $\BS_{3}$-module by $$v_1+v_5+v_9=\mu\circ_{{\rm I}}\mu+\mu\circ_{{\rm II}}\mu+\mu\circ_{{\rm III}}\mu=(x\mu y)\mu z+(z\mu x)\mu y+(y\mu z)\mu x.$$
Then we compute
{\allowdisplaybreaks
\begin{eqnarray*}
\tdu_{\{x\}}(v_1+v_5+v_9)&=&\{(x \dashv y) \dashv z+(z \vdash x)\dashv y+(y \dashv z) \vdash x, (x \dashv y) \dashv z+(z \vdash x)\dashv y \\
&& +(y \vdash z) \vdash x, (x \dashv y) \dashv z+(z \vdash x)\dashv y+(y \perp z) \vdash x\}\\
&=&\{\underline{(x \dashv y) \dashv z - (x \dashv z)\dashv y - x \dashv (y \dashv z)}, (x \dashv y) \dashv z - (x \dashv z)\dashv y \\
&& + x \dashv (z \dashv y),\underline{(x \dashv y) \dashv z - (x \dashv z)\dashv y - x \dashv (y \perp z)}\};\\
\tdu_{\{x,y\}}(v_1+v_5+v_9)&=&\{\underline{(x\perp y)\dashv z+(z\vdash x)\perp y+(y\dashv z)\perp x}\};\\
\tdu_{\{x,y,z\}}(v_1+v_5+v_9)&=&\{\underline{(x\perp y)\perp z+(z\perp x)\perp y+(y\perp z)\perp x}\},
\end{eqnarray*}
}
and other computations yield the same relations up to permutations.

Replacing the operation $\dashv$ by $\diamond$ and $\perp$ by $[,]$, then $[\,,\,]$ is skew-symmetric and the underlined relations are
\begin{eqnarray}
&[x,[y,z]]+[y,[z,x]]+[z,[x,y]]=0,&\notag \\
&x\diamond[y,z] = x \diamond (y \diamond z),\notag\\
&[x,y] \diamond z = [x \diamond z, y] + [x, y \diamond z],& \label{eq:trileib}\\
&(x \diamond y) \diamond z = x \diamond (y \diamond z) + (x \diamond z) \diamond y .&\notag
\end{eqnarray}
Then in particular $(A,\diamond)$ is a right Leibniz algebra. Since the duplicator of $\lie$ is $\leib$, the operad of the Leibniz algebra, we tentatively call the new algebra {\bf triLeibniz algebra}.
In summary, we obtain

\begin{prop}
The triplicator of $\lie$ is $\TriLeib$, the operad of the triLeibniz algebra.\mlabel{prop:triLeib}
\end{prop}

As we will see in Section~\mref{ss:dual}, $\TriLeib$ is precisely the Koszul dual of the operad $CTD=ComTriDend$ of the commutative tridendriform algebra, namely the {\bf Dual CTD algebra} in~\cite{Zi}.

We next show that $\TriLeib$ plays the same role for the triassociative algebra as the role of the Leibniz algebra for the diassociative algebra~\mcite{Fr2,LR}.

\begin{prop}
Let $(A,\dashv,\vdash,\perp)$ be an associative trialgebra. Define new binary operations by
$$x\diamond y:=x\dashv y-y\vdash x,\quad [x,y]:=x\perp y-y\perp x.$$
Then $(A,\diamond, [,])$ becomes a Leibniz trialgebra.\mlabel{prop:asleib}
\end{prop}

\begin{proof}
By definition, for any $x,y,z\in A$, we have
$$[x,y]\diamond z=[x,y]\dashv z-z\vdash [x,y]=(x\perp y-y\perp x)\dashv z-z\vdash(x\perp y-y\perp x)$$
and
\begin{eqnarray*}
& &[x \diamond z, y] + [x, y \diamond z]\\
&=&[x\dashv z-z\vdash x, y]+[x,y\dashv z-z\vdash y]\\
&=&(x\dashv z-z\vdash x)\perp y-y\perp(x\dashv z-z\vdash x)+x\perp(y\dashv z-z\vdash y)-(y\dashv z-z\vdash y)\perp x\\
& &
\end{eqnarray*}
Since
$$(x\perp y)\dashv z=x\perp(y\dashv z),\; (y\perp x)\dashv z=y\perp(x\dashv z),\; z\vdash(x\perp y)=(z\vdash x)\perp y,$$
$$z\vdash(y\perp x)=(z\vdash y)\perp x,\; (x\dashv z)\perp y=x\perp(z\vdash y),\; y\perp(z\vdash x)=(y\dashv z)\perp x$$
in a triassociative algebra, we have
$[x,y]\diamond z = [x \diamond z, y] + [x, y \diamond z].$
The other defining equations of the triLeibniz algebra can be proved in the same way.
\end{proof}

Moreover, we have the following commuting diagram.

\begin{equation*}
\xymatrix{Leib\ar@{->}[d]_{\diamond\to(\diamond,0)} & &Diass \ar@{->}_{-}[ll]\ar@{->}_{(\vdash,\dashv)\to(\vdash,\dashv,0)}[d] \\
TriLeib & & Triass \ar@{->}_{-}[ll] }
\end{equation*}
It would be interesting to consider the left adjoint of the functor defined in the bottom line of the above diagram, which could be called {\bf the universal envelope algebra of a triLeibniz algebra} just as in~\mcite{Lo2}.

\subsection{Operads, their duplicators and triplicators}
In this section, we study the relationship among a binary operad, its duplicator and its triplicator.
\subsubsection{Operads and their duplicators and triplicators}

For a given $\BS$-module $V$ concentrated in ariry $2$: $V=V(2)$. Let $i_V:V\to \mathcal{T}(V)$ denote the natural embedding to the free operad $\mathcal{T}(V)$. Let $\calp:=\mathcal{T}(V)/(R)$ be a binary operad and let $j_V:V\to \calp$ be $p_V\circ i_V$, where $p_V:\mathcal{T}(V)\to \calp$ is the operad projection. Similarly define the maps $i_{\du(V)}:\du(V)\to \mathcal{T}(\du(V))$ and operad morphism $p_{\du(V)}: \mathcal{T}(\du(V))\to \du(\calp)$ and $j_{\du(V)}:=p_{\du(V)}\circ i_{\du(V)}$, as well as the corresponding map and operad morphisms for $\tdu(V)$.

\begin{prop}\mlabel{prop:quotient}
Let $\calp=\mathcal{T}(V)/(R)$ be a binary operad.
\begin{enumerate}
\item
The linear map
\begin{equation}
\eta: \du(V) \to V, \quad
\svec{\gop}{u} \longmapsto \omega\ \text{for all } \svec{\gop}{u}\in \du(V), u \in \{\dashv,\vdash\}
\mlabel{eq:dumap}
\end{equation}
induces a unique operad morphism
$$ \tilde{\eta}:\du(\calp)\to \calp$$
such that $\tilde{\eta}\circ j_{\du(V)} = j_V\circ \eta.$
\mlabel{it:quotienta}
\item
The linear map
\begin{equation}
\zeta: \tdu(V) \to V, \quad
\svec{\gop}{u} \longmapsto \omega\ \text{for all } \svec{\gop}{u}\in \tdu(V), u \in \{\dashv,\vdash,\perp\}
\mlabel{eq:tdumap}
\end{equation}
induces a unique operad morphism
$$ \tilde{\zeta}:\tdu(\calp)\to \calp$$
such that $\tilde{\zeta}\circ j_{\tdu(V)} = j_V\circ \zeta.$
\mlabel{it:quotientb}
\item
There is a morphism $\rho:\tdu(\calp)\to\calp$ of operads that extends the linear map from $\tdu(V)$ to $V$ defined by
\begin{equation}
\svec{\gop}{\perp} \longmapsto \omega,\quad \svec{\gop}{u} \longmapsto 0,\quad \mbox{where} ~u \in \{\dashv,\vdash\}.\mlabel{tdumap2}
\end{equation}
\mlabel{it:quotientc}
\end{enumerate}
\end{prop}

\begin{proof}
Let $R$ be the set of locally homogeneous elements
$$
r_s:=\sum_i c_{s,i}\tau_{s,i}, \ c_{s,i}\in\bfk, \tau_{s,i} \in \bigcup_{t\in \mathfrak{R}} t(\genbas), \ 1\leq s\leq k,
$$
as given in Eq.(\mref{eq:pres}).

(\ref{it:quotienta}) By the universal property of the free operad $\mathcal{T}(\du(V))$ on the $\BS$-module $\du(V)$, the $\BS$-module morphism
$i_V\circ \eta: \du(V)\to \mathcal{T}(V)$
induces a unique operad morphism
$ \free{\eta}:\mathcal{T}(\du(V))\to \mathcal{T}(V)$
such that $i_{\du(V)}\circ \free{\eta} = i_V\circ \eta$.

For any $x\in \lin(r_s)$ and $1\leq s\leq k$, by the description of $\du_{x}(\tau_{s,i}))$ in Proposition~\mref{reppath} and the definition of $\eta$ in Eq.~(\mref{eq:dumap}), the element $\eta(\du(\tau_{s,i}))$ is obtained by replacing each decoration $\svec{\gop}{u}$ of the vertices of $\du(\tau_{s,i})$ by $\gop$, where $\gop \in V$ and $u \in \{\dashv, \vdash\}$. Thus $\free{\eta} (\du(\tau_{s,i})) = \tau_{s,i}$. Then we have
$$
\free{f}\left(\sum_{i}c_{s,i}\du_{x}(\tau_{s,i})\right) = \sum_ic_{s,i}\tau_{s,i} \equiv 0 \mod(R).
$$
By Eq.~(\mref{reprelation}), we see that $(\du(R))  \subseteq \ker (\eta)$. Thus there is a unique operad morphism $\tilde{\eta}: \du(\calp):=\mathcal{T}(\du(V))/(\du(R)) \to \calp:=\mathcal{T}(V)/(R)$ such that $\tilde{\eta}\circ p_{\du(V)}=p_V\circ \free{\eta}.$ We then have
$\tilde{\eta}\circ j_{\du(V)}=j_V \circ \eta$.
In summary, we have the following diagram in which each square commutes.
$$ \xymatrix{ \du(V) \ar^{\eta}[d] \ar^{i_{\du(V)}}[rr] && \mathcal{T}(\du(V)) \ar^{\free{\eta}}[d] \ar^{p_{\du(V)}}[rr] & & \du(\calp)
\ar^{\tilde{\eta}}[d] \\
V\ar^{i_V}[rr] && \mathcal{T}(V) \ar^{p_V}[rr] & & \calp
}
$$
Suppose $\tilde{\eta}':\du(\calp)\to \calp$ be another operad morphism such that $\tilde{\eta}'\circ j_{\du(V)}=j_V\circ \eta$. Then we have
$\tilde{\eta}'\circ j_{\du(V)}=\tilde{\eta}'\circ p_{\du(V)}\circ i_{\du(V)}$ and
$j_V\circ \eta=p_V\circ j_V\circ \eta=p_V\circ \free{\eta}\circ i_{\du(V)}.$ By the universal property of the free operad $\mathcal{T}(\du(V))$, we obtain
$\tilde{\eta}'\circ p_{\du(V)} = p_V\circ \free{\eta} = \tilde{\eta}\circ p_{\du(V)}.$
Since $p_{\du(V)}$ is surjective, we obtain $\tilde{\eta}'=\tilde{\eta}.$ This proves the uniqueness of $\tilde{\eta}.$

(\ref{it:quotientb}) The proof is similar to the proof of Item (\ref{it:quotienta}).

(\ref{it:quotientc}) By the description of $\tdu_{\{x\}}(\tau_{s,i})$ in Proposition~\mref{treppath}, $\rho(\tdu_{\{x\}}(\tau_{s,i}))$ is obtained by replacing $\svec{\gop}{u}$ by $\rho(\svec{\gop}{u})$.
Since $\rho(\svec{\gop}{\dashv})=0$, $\rho(\svec{\gop}{\vdash})=0$ and $\rho(\svec{\gop}{\perp})= \omega$, it is easy to see that if $J \neq Lin(\tau)$, then
$\rho(\sum_{i}c_{s,i}\tdu_{J}(\tau_{s,i})) =\sum_ic_{s,i}\tau_{s,i} = 0,
$
and, if $J = Lin(\tau)$, then $\rho(\sum_{i}c_{s,i}\tdu_{Lin(\tau)}(\tau_{s,i}))= \sum_ic_{s,i}\tau_{s,i} \equiv 0 \mod(R).
$
Thus $\rho(\tdu(R))\subseteq R$ and $\rho$ induces the desired operad morphism.
\end{proof}

\subsubsection{Relationship between duplicators and triplicators of a binary operad}

The following result relates the duplicator and the triplicator of a binary algebraic operad.
\begin{prop}
Let $\calp=\mathcal{T}(V)/(R)$ be a binary algebraic operad. There is a morphism of operads from $\tdu(\calp)$ to $\du(\calp)$ that extends the linear map defined by
\begin{equation}
\svec{\gop}{\dashv}\to \svec{\gop}{\dashv},\quad  \svec{\gop}{\vdash}\to \svec{\gop}{\vdash},\quad \svec{\gop}{\perp}\to 0,\quad \gop\in\gensp.\mlabel{eq:sutsze}
\end{equation}
\mlabel{it:rests2}
\end{prop}

\begin{proof}

The linear map $\phi:\tdu(V)\to \du(V)$ defined by Eq.(\mref{eq:sutsze}) is $\BS_{2}$-equivariant. Hence it induces a morphism of the free operads $\phi: \mathcal{T}(\tdu(V))\to \mathcal{T}(\du(V))$ which, by composing with the quotient map, induces the morphism of operads
$$ \phi: \mathcal{T}(\tdu(V))\to \du(\calp)=\mathcal{T}(\du(V))/(\du(R)).$$
Let $\tdu_J(r)\in \tdu(R)$ be one of the generators of $(\tdu(R))$ with $r=\sum_i c_i \tau_i \in R$ in Eq.~(\mref{eq:pres}) and $\varnothing\neq J\subseteq \lin(r)$.
If $J$ is the singleton $\{x\}$ for some $x\in \lin(r)$, then by the description of $\tdu_{\{x\}}(\tau_i)$ in Proposition~\mref{treppath}, $\phi(\tdu_{\{x\}}(\tau))$ is obtained by keeping all the $\svec{\gop}{\dashv}$ and $\svec{\gop}{\vdash}$, and by replacing all $\svec{\gop}{\perp}, \gop\in \gensp$ by zero. Thus in Case (b) of Proposition~\mref{treppath} we have $\phi(\tdu_{\{x\}}(\tau_i))=\du_x(\tau_i)$. Also Case (a)(iii) cannot occur for the singleton $\{x\}$. Thus in Case (a) of Proposition~\mref{treppath}, we also have
$\phi(\tdu_{\{x\}}(\tau_i))=\du_x(\tau_i)$. Thus $\phi(\tdu_{\{x\}}(r))=\du_x(r)$ and hence is in $\du(R)$.

If $J$ contains more than one element, then at least one of the vertices of $\tdu_J(\tau_i)$ is $\svec{\gop}{\perp}$ and hence the corresponding vertex of $\phi(\tsu_J(\tau_i))$ is zero. Thus we have
$\phi(\tdu_{J}(\tau_{i})) = 0$, $\phi(\tdu_J(r))=0$ and hence $\phi(\tdu_J(R))=0$.
Thus, for any $J\neq \varnothing$ and $r\in R$, we have $\phi(\tdu_J(r)) \in \du(R)$ and hence $\phi(\tdu(R))$ is a subset of $\du(R)$.

In summary, we have $\phi((\tdu(R))\subseteq \du(R)$. Thus the morphism $\phi:\mathcal{T}(\tdu(\gensp))\to \du(\calp)$ induces a morphism
$\phi:   \tdu(\calp) \to \du(\calp).$
\end{proof}


If we take $\calp$ to be the operad of the associative algebra, then we obtain the following result of Loday and Ronco~\mcite{LR}:
\begin{coro}
Let $(A,\dashv,\vdash)$ be an associative dialgebra. Then $(A,\dashv,\vdash,0)$ is an associative trialgebra, where 0 denotes the trivial product.
\end{coro}


\section{Duality of replicators with successors and Manin products}
\mlabel{sec:mp}
The similarity between the definitions of the replicators and successors~\mcite{BBGN} suggests that there is a close relationship between the two constructions. We show that this is indeed the case. More precisely, taking the replicator of a binary quadratic operad is in Koszul dual with taking the successor of the dual operad.
This in particular allows us to identify the duplicator (resp. triplicator) of a binary quadratic operad $\calp$ with the Manin white product of $\Perm$ (resp. {\it ComTrias}) with $\calp$, providing an easy way to compute these white products. Since it is shown in~\mcite{GK2} that taking di-Var and tri-Var are also isomorphic to taking these Manin products, taking duplicator (resp. triplicator) is isomorphic to taking di-Var (resp. tri-Var) other than the case of free operads.

\subsection{The duality of replicators with successors} \mlabel{ss:dual}

Let $\opd=\mathcal{T}(\gensp)/(R)$ be a binary quadratic operad. Then with the notations in Section~\mref{ss:exam}, we have $\mathcal{T}(\gensp)(3)=3 \gensp \ot \gensp=\bigoplus_{u\in \{\mathrm{I,II,III}\}} V\circ_u V$.

\begin{prop}
Let $\bfk$ be an infinite field. Let $W$ be a nonzero $\BS$-submodule of $3\gensp \ot \gensp$. Then there is a basis $\{e_1,\cdots,e_n\}$ of $\gensp$ such that the restriction to $W$ of the coordinate projections
$$ p_{i,j,u}: 3\gensp \ot \gensp = \bigoplus_{1\leq k,\ell\leq n,v\in \{\mathrm{I,II,III}\}} \bfk\, e_k\circ_v e_\ell \to \bfk\, e_i\circ_u e_j,$$
are nonzero and hence surjective for all $1\leq i, j\leq n$ and $u\in \{\mathrm{I,II,III}\}$.
\mlabel{pp:generic}
\end{prop}

\begin{proof}
Fix a $0\neq w\in W$ and write $w=w_{\mathrm{I}}+w_{\mathrm{II}}+w_{\mathrm{III}}$ with $w_u\in V\circ_u V, u\in \{\mathrm{I,II,III}\}$. Then at least one of the three terms is nonzero. Since $W$ is an $\BS$-module and $(w_{u})^{(123)} = w_{u+I}$ (where $\mathrm{III}+\mathrm{I}$ is taken to be $\mathrm{I}$), we might assume that $w\in W$ is chosen so that $w_\mathrm{I}\neq 0$. Fix a basis $\{v_1,\cdots,v_n\}$ of $V$. Then there are $c_{ij}\in \bfk, 1\leq i, j\leq n,$ that are not all zero such that $w_\mathrm{I}=\sum_{1\leq i,j\leq n} c_{ij}u_i\circ_\mathrm{I} u_j$.

Consider the set of polynomials
$$ f_{k\ell}(x_{rs}):=f_{k\ell} (\{x_{rs}\}):=\sum_{1\leq i, j\leq n} c_{ij}x_{ik}x_{j\ell} \in \bfk [x_{rs}\,|\, 1\leq r,s\leq n], \quad 1\leq k, \ell\leq n.$$
Then the polynomial
$\prod_{1\leq k,\ell\leq n} f_{k\ell}(x_{rs})$
is nonzero since at least one of $c_{ij}$ is nonzero, giving a monomial
$\prod_{1\leq r, s\leq n} c_{ij} x_{i r}x_{j s}$
in the product with nonzero coefficient. Hence the product
$$f(x_{rs}):=\det(x_{rs})\prod_{1\leq k,\ell\leq n} f_{k\ell}(x_{rs})$$
is nonzero since $\det(x_{rs}):=\prod_{\sigma\in \BS_n} x_{1 \sigma(1)}\cdots x_{n \sigma(n)}$ is also a nonzero polynomial. Thus, by our assumption that $\bfk$ is an infinite field, there are $d_{rs}\in \bfk, 1\leq r, s\leq n,$ such that
$f(d_{rs})\neq 0$. Thus $D:=(d_{rs})\in M_{n\times n}(\bfk)$ is invertible and $f_{k\ell}(d_{rs})\neq 0, 1\leq k, \ell\leq n$.

Fix such a matrix $D=(d_{rs})$ and define $$(e_1,\cdots,e_n)^T:=D^{-1}(v_1,\cdots,v_n)^T.$$
Then $\{e_1,\cdots, e_n\}$ is a basis of $V$ and  $v_i=\sum_{k=1}^n d_{ik}e_k$. Further

$$
w_{\mathrm{I}}= \sum_{1\leq i,j\leq n} c_{ij}v_i\circ_\mathrm{I} v_j
= \sum_{1\leq i,j\leq n} c_{ij}
\left( \sum_{1\leq k,\ell n} d_{ik}d_{j\ell} e_k \circ_\mathrm{I} e_{\ell} \right)
=
\sum_{1\leq k,\ell\leq n} \left (\sum_{1\leq i, j\leq n} c_{ij}d_{ik}d_{j\ell}\right) e_k \circ_\mathrm{I} e_{\ell}.
$$
The coefficients are $f_{k\ell}(d_{rs})$ and are nonzero by the choice of $D$. Thus $p_{i,j,\mathrm{I}}(w) =p_{i,j,\mathrm{I}}(w_\mathrm{I})$ is nonzero and hence $p_{i,j,\mathrm{I}}(W)$ is onto for all $1\leq i, j\leq n$.

Since $W$ is an $\BS$-module, we have $w^{(123)}\in W$ and $(w^{(123)})_\mathrm{II}=(w_\mathrm{I})^{(123)}$.
Thus $p_{i,j,\mathrm{II}}(w^{(123)}) =p_{i,j,\mathrm{II}}((w_\mathrm{I})^{(123)})$ is nonzero and hence $p_{i,j,\mathrm{II}}(W)$ is onto for all $1\leq i, j\leq n$. By the same argument, $p_{i,j,\mathrm{III}}(W)$ is onto for all $1\leq i, j\leq n$, completing the proof.
\end{proof}

\begin{lemma}
Let $W$ be a nonzero $\BS$-submodule of $3V\ot V$ and let $\{e_1,\cdots,e_n\}$ be a basis as chosen in Proposition~\mref{pp:generic}. Let $\{r_1,\cdots,r_m\}$ be a basis of $U$ and write
$$ r_k = \sum_{1\leq i, j\leq n} c^\ell_{iju} e_i\circ_u e_j, \quad c^k_{iju}\in \bfk, 1\leq i,j\leq n, u\in \{\mathrm{I,II,III}\}, 1\leq k\leq m.$$
Then for each $1\leq i,j\leq n$ and $u\in \{\mathrm{I,II,III}\}$, there is $1\leq k\leq m$, such that $c^\ell_{iju}$ is not zero.
\mlabel{lem:genbas}
\end{lemma}

\begin{proof}
Suppose there is $1\leq i,j\leq n$ and $u\in \{\mathrm{I,II,III}\}$ such that $c^k_{iju}=0$ for all $1\leq k\leq m$. Then $p_{iju}(r_k)=0$ and hence $p_{iju}(W)=0$. This contradicts Proposition~\mref{pp:generic}.
\end{proof}

Let $\opd=\mathcal{T}(\gensp)/ (R)$ be a binary quadratic operad. Fix a $\bfk$-basis $\{e_{1},e_{2},\cdots,e_{n}\}$ for $(R)$. The space $\mathcal{T}(\gensp)(3)$ is spanned by the basis $\{e_{i} \circ_{u} e_{j}~|~ 1 \leq i,j \leq n, u \in \{\rm I,II,III\}\}$. Thus if $f \in \mathcal{T}(V)(3)$, we have
$$
f = \sum_{i,j}a_{i,j} e_{i} \circ_{\rm I} e_{j} + \sum_{i,j}b_{i,j}e_{i} \circ_{\rm II} e_{j} +\sum_{i,j}c_{i,j}e_{i} \circ_{\rm III} e_{j}.
$$
Then we can take the relation space $(R) \subset \mathcal{T}(V)(3)$ to be generated by $m$ linearly independent relations
\begin{equation}
R = \left\{ f_{k} = \sum_{i,j}a_{i,j}^{k} e_{i} \circ_{\rm I} e_{j} + \sum_{i,j}b_{i,j}^{k}e_{i} \circ_{\rm II} e_{j} +\sum_{i,j}c_{i,j}^{k}e_{i} \circ_{\rm III} e_{j}~\Big|~ 1\leq k \leq m \right\}.
\mlabel{relation}
\end{equation}

We state the following easy fact for later applications.
\begin{lemma}
Let $f_i, 1\leq i\leq m,$ be a basis of $(R)$. Then
$\{\bsu_x(f_i)\,|\, x\in \lin(f_i), 1\leq i\leq m\}$ is a linear spanning set of $(\bsu(R))$ and $\{\du_x(f_i)\,|\, x\in \lin(f_i), 1\leq i\leq m\}$ is a linear spanning set of $(\du(R))$.
\mlabel{lem:smod}
\end{lemma}
\begin{proof}
Let $L$ be the linear span of $\{\bsu_x(f_i)\,|\, x\in \lin(f_i), 1\leq i\leq m\}$. Then from $\bsu_x(f_i)\in (\bsu(R))$ we obtain $L\subseteq (\bsu(R))$. On the other hand, by \cite[Lemma 2.6]{BBGN}, $L$ is already an $\BS$-submodule. Thus from $\bsu(R)\subseteq L$ we obtain $(\bsu(R))_\BS\subseteq L.$
The proof for $(\du(R))$ is the same.
\end{proof}

For the finite dimensional $\mathbb{S}_{2}$-module $V$, we define its Czech dual $V^{\vee} = V^{\ast} \otimes sgn_{2}$. There is a natural pairing with respect to this duality given by:
$$
\langle,\rangle : \mathcal{T}(V^{\vee})(3) \otimes \mathcal{T}(V)(3) \longrightarrow \bfk,
$$
$$
\langle e_{i}^{\vee} \circ_{u} e_{j}^{\vee},e_{k} \circ_{v} e_{\ell} \rangle = \delta_{(i,k)}\delta_{(j,\ell)}\delta_{(u,v)} \in \bfk.
$$
We denote by $R^{\perp}$ the annihilator of $R$ with respect to this pairing. Given relations as in Eq.~(\ref{relation}), we can express a basis of $(R^{\perp})$ as
\begin{equation}
R^{\perp} = \left\{ g_{\ell} = \sum_{i,j}\alpha_{i,j}^{\ell} e_{i}^{\vee} \circ_{\rm I} e_{j}^{\vee} + \sum_{i,j}\beta_{i,j}^{\ell}e_{i}^{\vee} \circ_{\rm II} e_{j}^{\vee} +\sum_{i,j}\gamma_{i,j}^{\ell}e_{i}^{\vee} \circ_{\rm III} e_{j}^{\vee}~|~ 1\leq \ell \leq 3n^{2} - m \right\},
\mlabel{eq:prel}
\end{equation}
where, for all $k$ and $\ell$, we have
\begin{equation}
\sum_{i,j}a_{i,j}^{k} \alpha_{i,j}^{\ell} + \sum_{i,j}b_{i,j}^{k} \beta_{i,j}^{\ell}  + \sum_{i,j}c_{i,j}^{k}\gamma_{i,j}^{\ell} = 0.
\mlabel{perp}
\end{equation}
Further for any $(x_{i,j},y_{i,j},z_{i,j})\in \bfk^3, 1\leq i,j\leq n$, if
\begin{equation*}
\sum_{i,j}a_{i,j}^{k}x_{i,j}  + \sum_{i,j}b_{i,j}^{k} y_{i,j}  + \sum_{i,j}c_{i,j}^{k}z_{i,j} = 0 \ \text{ for all } 1\leq k\leq m,
\end{equation*}
then $\sum_{i,j} x_{i,j}e_{i}^{\vee} \circ_{\rm I} e_{j}^{\vee} + \sum_{i,j}y_{i,j}e_{i}^{\vee} \circ_{\rm II} e_{j}^{\vee} +\sum_{i,j}z_{i,j}e_{i}^{\vee} \circ_{\rm III} e_{j}^{\vee}$ is in $R^{\perp}$ and hence is of the form
$\sum\limits_{\ell=1}^{3n^2-m} d_\ell g_\ell$ for some $d_\ell\in \bfk$. Thus
\begin{equation}
(x_{i,j},y_{i,j},z_{i,j}) = \sum_{\ell=1}^{3n^2-m}d_{\ell}
\left(\alpha_{i,j}^{\ell},  \beta_{i,j}^{\ell}, \gamma_{i,j}^{\ell}\right).
\mlabel{eq:anninv}
\end{equation}

By Proposition~\ref{reppath}, we have
\begin{eqnarray}
&\du_x(e_i\circ_{\rm I} e_j)=\left\{ \svec{e_i}{\dashv}\circ_{\rm I} \svec{e_j}{\dashv}\right\},
\du_x(e_i\circ_{\rm II} e_j) =\left \{
\svec{e_i}{\vdash}\circ_{\rm II} \svec{e_j}{\rep}\right\},
\du_x(e_i\circ_{\rm III} e_j) = \left \{
\svec{e_i}{\dashv}\circ_{\rm III} \svec{e_j}{\vdash}\right\},& \notag\\
& \du_y(e_i\circ_{\rm I} e_j)=\left\{ \svec{e_i}{\dashv}\circ_{\rm I} \svec{e_j}{\vdash}\right\},
\du_y(e_i\circ_{\rm II} e_j) =\left \{
\svec{e_i}{\dashv}\circ_{\rm II} \svec{e_j}{\dashv}\right\},
\du_y(e_i\circ_{\rm III} e_j) = \left \{
\svec{e_i}{\vdash}\circ_{\rm III} \svec{e_j}{\rep} \right\},& \label{eq:dubase}\\
& \du_z(e_i\circ_{\rm I} e_j)=\left\{ \svec{e_i}{\vdash}\circ_{\rm I} \svec{e_j}{\rep}\right\},
\du_z(e_i\circ_{\rm II} e_j) =\left \{
\svec{e_i}{\dashv}\circ_{\rm II} \svec{e_j}{\vdash}\right\},
\du_z(e_i\circ_{\rm III} e_j) = \left \{
\svec{e_i}{\dashv}\circ_{\rm III} \svec{e_j}{\dashv} \right\},& \notag
\end{eqnarray}
where $\svec{e_{i}}{\vdash} \circ_{\rm u} \svec{e_{j}}{\rep} :=\left \{\svec{e_{i}}{\vdash} \circ_{\rm u} \svec{e_{j}}{\vdash},\svec{e_{i}}{\vdash} \circ_{\rm u} \svec{e_{j}}{\dashv}\right\}, u \in\{\rm I,II,III\}$.

Let $\bsu(\calp^!)$ be the bisuccessor of the dual operad $\calp^!$ recalled in Definition~\mref{de:bsu}. Then we also have
{\small
\begin{eqnarray}
&\bsu_x(e^{\vee}_i\circ_{\rm I} e^{\vee}_j)=\left\{ \svec{e^{\vee}_i}{\prec}\circ_{\rm I} \svec{e^{\vee}_j}{\prec}\right\},
\bsu_x(e^{\vee}_i\circ_{\rm II} e^{\vee}_j) =\left \{
\svec{e^{\vee}_i}{\succ}\circ_{\rm II} \svec{e^{\vee}_j}{\star}\right\},
\bsu_x(e^{\vee}_i\circ_{\rm III} e^{\vee}_j) = \left \{
\svec{e^{\vee}_i}{\prec}\circ_{\rm III} \svec{e^{\vee}_j}{\succ}\right\},& \notag\\
& \bsu_y(e^{\vee}_i\circ_{\rm I} e^{\vee}_j)=\left\{ \svec{e^{\vee}_i}{\prec}\circ_{\rm I} \svec{e^{\vee}_j}{\succ}\right\},
\bsu_y(e^{\vee}_i\circ_{\rm II} e^{\vee}_j) =\left \{
\svec{e^{\vee}_i}{\prec}\circ_{\rm II} \svec{e^{\vee}_j}{\prec}\right\},
\bsu_y(e^{\vee}_i\circ_{\rm III} e^{\vee}_j) = \left \{
\svec{e^{\vee}_i}{\succ}\circ_{\rm III} \svec{e^{\vee}_j}{\star} \right\},& \label{eq:bsubase}\\
& \bsu_z(e^{\vee}_i\circ_{\rm I} e^{\vee}_j)=\left\{ \svec{e^{\vee}_i}{\succ}\circ_{\rm I} \svec{e^{\vee}_j}{\star}\right\},
\bsu_z(e^{\vee}_i\circ_{\rm II} e^{\vee}_j) =\left \{
\svec{e^{\vee}_i}{\prec}\circ_{\rm II} \svec{e^{\vee}_j}{\succ}\right\},
\bsu_z(e^{\vee}_i\circ_{\rm III} e^{\vee}_j) = \left \{
\svec{e^{\vee}_i}{\prec}\circ_{\rm III} \svec{e^{\vee}_j}{\prec} \right\},& \notag
\end{eqnarray}
}
where $\star = \prec + \succ.$
\begin{theorem}
Let $\bfk$ be an infinite field. Let $\mathcal{P} = {\mathcal T}(V)/(R)$ be a binary quadratic operad. Then
$$
\du(\mathcal{P})^{!} =  \bsu(\mathcal{P}^{!})
$$
if and only if $R\neq 0$.
\mlabel{thm:dusu}
\end{theorem}

\begin{proof}
For the if part, let $\opd=\mathcal{T}(\gensp)/ (R)$ be a binary quadratic operad with $R\neq 0$. Take $W=(R)$ in Proposition~\mref{pp:generic} and fix a $\bfk$-basis $\{e_{1},e_{2},\cdots,e_{n}\}$ of $V$ as in the proposition.
Let $f_k, 1\leq k\leq m,$ be the basis of $(R)$ as defined in Eq.~(\mref{relation}).

By Eq.~(\ref{eq:dubase}), we have
\begin{eqnarray*}
\du_{x}(f_{k}) =\left\{\sum_{i,j}a_{i,j}^{k} \svec{e_{i}}{\dashv} \circ_{\rm I} \svec{e_{j}}{\dashv} + \sum_{i,j}b_{i,j}^{k}\svec{e_{i}}{\vdash} \circ_{\rm II} \svec{e_{j}}{\rep} +\sum_{i,j}c_{i,j}^{k}\svec{e_{i}}{\dashv} \circ_{\rm III} \svec{e_{j}}{\vdash}\right\},
\end{eqnarray*}
\begin{eqnarray*}
\du_{y}(f_{k})=\left\{\sum_{i,j}a_{i,j}^{k} \svec{e_{i}}{\dashv} \circ_{\rm I} \svec{e_{j}}{\vdash} + \sum_{i,j}b_{i,j}^{k}\svec{e_{i}}{\dashv} \circ_{\rm II} \svec{e_{j}}{\dashv} +\sum_{i,j}c_{i,j}^{k}\svec{e_{i}}{\vdash} \circ_{\rm III} \svec{e_{j}}{\rep}\right\},
\end{eqnarray*}
\begin{eqnarray*}
\du_{z}(f_{k})=\left\{\sum_{i,j}a_{i,j}^{k} \svec{e_{i}}{\vdash} \circ_{\rm I} \svec{e_{j}}{\rep} + \sum_{i,j}b_{i,j}^{k}\svec{e_{i}}{\dashv} \circ_{\rm II} \svec{e_{j}}{\vdash} +\sum_{i,j}c_{i,j}^{k}\svec{e_{i}}{\dashv} \circ_{\rm III} \svec{e_{j}}{\dashv}\right\}.
\end{eqnarray*}
From Eq.~(\ref{eq:bsubase}), we similarly obtain
$$
\bsu_{x}(g_{\ell}) = \left
\{\sum_{i,j}\alpha_{i,j}^{\ell} \svec{e^{\vee}_{i}}{\prec} \circ_{\rm I} \svec{e^{\vee}_{j}}{\prec} + \sum_{i,j}\beta_{i,j}^{\ell}\svec{e^{\vee}_{i}}{\succ} \circ_{\rm II} \svec{e^{\vee}_{j}}{\star} +\sum_{i,j}\gamma_{i,j}^{\ell}\svec{e^{\vee}_{i}}{\prec} \circ_{\rm III} \svec{e^{\vee}_{j}}{\succ}\right\},
$$

$$
\bsu_{y}(g_{\ell}) = \left\{\sum_{i,j}\alpha_{i,j}^{\ell} \svec{e^{\vee}_{i}}{\prec} \circ_{\rm I} \svec{e^{\vee}_{j}}{\succ} + \sum_{i,j}\beta_{i,j}^{\ell}\svec{e^{\vee}_{i}}{\prec} \circ_{\rm II} \svec{e^{\vee}_{j}}{\prec} +\sum_{i,j}\gamma_{i,j}^{\ell}\svec{e^{\vee}_{i}}{\succ} \circ_{\rm III} \svec{e^{\vee}_{j}}{\star}\right\},
$$

$$
\bsu_{z}(g_{\ell}) = \left\{\sum_{i,j}\alpha_{i,j}^{\ell} \svec{e^{\vee}_{i}}{\succ} \circ_{\rm I} \svec{e^{\vee}_{j}}{\star} + \sum_{i,j}\beta_{i,j}^{\ell}\svec{e^{\vee}_{i}}{\prec} \circ_{\rm II} \svec{e^{\vee}_{j}}{\succ} +\sum_{i,j}\gamma_{i,j}^{\ell}\svec{e^{\vee}_{i}}{\prec} \circ_{\rm III} \svec{e^{\vee}_{j}}{\prec}\right\}.
$$
By Lemma~\mref{lem:smod}, we have
$$
(\du(R)) = \sum_{k=1}^m\bfk\du(f_{k}) = \sum_{k}\left(\bfk\du_{x}(f_{k})+ \bfk\du_{y}(f_{k}) + \bfk\du_{z}(f_{k})\right),
$$
$$
\bsu(R^{\perp}) = \sum_{\ell=1}^{3n^2-m}\bfk \bsu(g_{\ell}) = \sum_{\ell=1}^{3n^2-m}\left(\bfk(\bsu_{x}(g_{\ell}) + \bfk\bsu_{y}(g_{\ell}) + \bfk\bsu_{z}(g_{\ell})\right).
$$

To reach our conclusion, it suffices to show the equality $(\du(R)^{\perp}) = (\bsu(R^{\perp}))$ of $\BS$-modules under the condition $R\neq 0$. For all $1\leq k\leq m$ and $1\leq \ell\leq 3n^2-m$, by Eq.~(\ref{perp}), we have
$$
\langle \bsu_{p}(g_{\ell}), \du_{q}(f_{k})\rangle = 0,~\mbox{where}~p,q \in \{x,y,z\} .
$$
Thus $\langle \bsu(g_{\ell}), \du(f_{k})\rangle = 0$ and hence $ \bsu(R^{\perp}) \subset \du(R)^{\perp}$, implying that $(\bsu(R^\perp))\subseteq (\du(R)^\perp)$. On the other hand, if
$$
h = \sum_{i,j,u,v}x_{i,j,u,v} \svec{e_{i}^{\vee}}{u} \circ_{\rm I} \svec{e_{j}^{\vee}}{v} + \sum_{i,j,u,v}y_{i,j,u,v}\svec{e_{i}^{\vee}}{u} \circ_{{\rm II}} \svec{e_{j}^{\vee}}{v} + \sum_{i,j,u,v}z_{i,j,u,v}\svec{e_{i}^{\vee}}{u} \circ_{{\rm III}} \svec{e_{j}^{\vee}}{v}
$$
is in $\du(R)^{\perp}$, where $u,v \in \{\prec,\succ\}$. Then for all $1\leq k\leq m$, we have
$$
\langle h, \du_{x}(f_{k})\rangle = 0, \langle h, \du_{y}(f_{k})\rangle = 0, \langle h, \du_{z}(f_{k})\rangle = 0.
$$

Since $R\neq 0$, by Proposition~\mref{pp:generic}, for any fixed $i_{0}, j_{0} \in \{1,2,\cdots,n\}$, there exists $1\leq k_{0}\leq m$, such that $b_{i_{0},j_{0}}^{k_{0}} \neq 0$. Then, for any $k$, by the definition of $\du_x$ we see that the relations
\begin{eqnarray*}
F_1:=\sum_{i,j}a_{i,j}^{k} \svec{e_{i}}{\dashv} \circ_{\rm I} \svec{e_{j}}{\dashv} + b_{i_{0},j_{0}}^{k} \svec{e_{i}}{\vdash} \circ_{\rm II} \svec{e_{j}}{\dashv} + \sum_{i\neq i_{0},j\neq j_{0}}b_{i,j}^{k}\svec{e_{i}}{\vdash} \circ_{\rm II} \svec{e_{j}}{\dashv} +\sum_{i,j}c_{i,j}^{k}\svec{e_{i}}{\dashv} \circ_{\rm III} \svec{e_{j}}{\vdash},\\
F_2:=\sum_{i,j}a_{i,j}^{k} \svec{e_{i}}{\dashv} \circ_{\rm I} \svec{e_{j}}{\dashv} + b_{i_{0},j_{0}}^{k} \svec{e_{i}}{\vdash} \circ_{\rm II} \svec{e_{j}}{\vdash} + \sum_{i\neq i_{0},j\neq j_{0}}b_{i,j}^{k}\svec{e_{i}}{\vdash} \circ_{\rm II} \svec{e_{j}}{\dashv} +\sum_{i,j}c_{i,j}^{k}\svec{e_{i}}{\dashv} \circ_{\rm III} \svec{e_{j}}{\vdash},\\
\end{eqnarray*}
are in $\du_{x}(f_{k})$.
Thus, for $1\leq k\leq m$, we obtain
\begin{eqnarray*}
\sum_{i,j}a_{i,j}^{k}x_{i,j,\prec,\prec} + b_{i_{0},j_{0}}^{k}y_{i_{0},j_{0},\succ,\prec} + \sum_{i\neq i_{0},j\neq j_{0}}b_{i,j}^{k}y_{i,j,\succ,\prec} +\sum_{i,j}c_{i,j}^{k}z_{i,j,\prec,\succ}=\langle h,F_1\rangle =0,\\
\sum_{i,j}a_{i,j}^{k}x_{i,j,\prec,\prec} + b_{i_{0},j_{0}}^{k}y_{i_{0},j_{0},\succ,\succ} +\sum_{i\neq i_{0},j\neq j_{0}}b_{i,j}^{k}y_{i,j,\succ,\prec} +\sum_{i,j}c_{i,j}^{k}z_{i,j,\prec,\succ}=\langle h,F_2\rangle =0.
\end{eqnarray*}
Comparing the two equations and applying $b_{i_0,j_0}^{k_0}\neq 0$, we obtain $y_{i_{0},j_{0},\succ,\succ} = y_{i_{0},j_{0},\succ,\prec}$ for all $1\leq i_0, j_0\leq n$.
From the second equation and Eq.~(\ref{eq:anninv}), we also have
$$
(x_{i,j,\prec,\prec}, y_{i,j,\succ,\prec},z_{i,j,\prec,\succ}) =\sum_{\ell=1}^{3n^2-m}d_\ell (\alpha_{i,j}^\ell, \beta_{i,j}^\ell, \gamma_{i,j}^\ell),
$$
for some $d_\ell\in \bfk$.
Thus we obtain
\begin{eqnarray*}
h_x&:=&\sum_{i,j}x_{i,j,\prec,\prec} \svec{e_{i}^{\vee}}{\prec} \circ_{\rm I} \svec{e_{j}^{\vee}}{\prec} + \sum_{i,j}y_{i,j,\succ,\prec}\svec{e_{i}^{\vee}}{\succ} \circ_{{\rm II}} \svec{e_{j}^{\vee}}{v} + \sum_{i,j}y_{i,j,\succ,\succ}\svec{e_{i}^{\vee}}{\succ} \circ_{{\rm II}} \svec{e_{j}^{\vee}}{v} +
\sum_{i,j}z_{i,j}\svec{e_{i}^{\vee}}{\prec} \circ_{{\rm III}} \svec{e_{j}^{\vee}}{\succ}\\
&=&\sum_{i,j}x_{i,j,\prec,\prec} \svec{e_{i}^{\vee}}{\prec} \circ_{\rm I} \svec{e_{j}^{\vee}}{\prec} + \sum_{i,j}y_{i,j,\succ,\prec}\left(\svec{e_{i}^{\vee}}{\succ} \circ_{{\rm II}} \svec{e_{j}^{\vee}}{v} +\svec{e_{i}^{\vee}}{\succ} \circ_{{\rm II}} \svec{e_{j}^{\vee}}{v}\right) +
\sum_{i,j}z_{i,j}\svec{e_{i}^{\vee}}{\prec} \circ_{{\rm III}} \svec{e_{j}^{\vee}}{\succ}\\
&=&\sum_{\ell=1}^{3n^2-m} d_\ell \left(\sum_{i,j}\alpha_{i,j}^{\ell} \svec{e_{i}^{\vee}}{\prec} \circ_{\rm I} \svec{e_{j}^{\vee}}{\prec}
 +\sum_{i,j}\beta_{i,j}^{\ell}\left(\svec{e_{i}^{\vee}}{\succ} \circ_{\rm II} \svec{e_{j}^{\vee}}{\prec}
+\svec{e_{i}^{\vee}}{\succ} \circ_{\rm II} \svec{e_{j}^{\vee}}{\succ}\right) +\sum_{i,j}\gamma_{i,j}^{\ell}\svec{e_{i}^{\vee}}{\prec} \circ_{\rm III} \svec{e_{j}^{\vee}}{\succ}\right).
\end{eqnarray*}
This is in $\sum\limits_{\ell=1}^{3n^2-m} \bfk \bsu_x(g_\ell).$
By the same argument, we find that
$$h_y:=\sum_{i,j}x_{i,j,\prec,\succ} \svec{e_{i}^{\vee}}{\prec} \circ_{\rm I} \svec{e_{j}^{\vee}}{\succ} + \sum_{i,j}y_{i,j,\prec,\prec}\svec{e_{i}^{\vee}}{\prec} \circ_{{\rm II}} \svec{e_{j}^{\vee}}{\prec} + \sum_{i,j,v}z_{i,j,\succ,\prec}\svec{e_{i}^{\vee}}{\succ} \circ_{{\rm III}} \svec{e_{j}^{\vee}}{\prec}  + \sum_{i,j,\succ,\succ} z_{i,j,\succ,\succ}\svec{e_{i}^{\vee}}{\succ} \circ_{{\rm III}} \svec{e_{j}^{\vee}}{\succ}$$
is in $\sum\limits_{\ell=1}^{3n^2-m} \bfk \bsu_y(g_\ell)$ and
$$h_z:=\sum_{i,j,\succ,\prec}x_{i,j,\succ,v} \svec{e_{i}^{\vee}}{\succ} \circ_{\rm I} \svec{e_{j}^{\vee}}{\prec} +
\sum_{i,j,\succ,\succ}x_{i,j,\succ,\succ} \svec{e_{i}^{\vee}}{\succ} \circ_{\rm I} \svec{e_{j}^{\vee}}{\succ} + \sum_{i,j}y_{i,j,\prec,\succ}\svec{e_{i}^{\vee}}{\prec} \circ_{{\rm II}} \svec{e_{j}^{\vee}}{\succ} + \sum_{i,j}z_{i,j,\prec,\prec}\svec{e_{i}^{\vee}}{\prec} \circ_{{\rm III}} \svec{e_{j}^{\vee}}{\prec}$$
is in $\sum\limits_{\ell=1}^{3n^2-m} \bfk \bsu_z(g_\ell)$.
Note that $h=h_x+h_y+h_z$. Thus in summary, we find that
$h$ is in
$$\sum_{\ell}\bfk \bsu_{x}(g_{\ell})+\sum_{\ell}\bfk \bsu_{y}(g_{\ell})+\sum_{\ell} \bfk \bsu_{z}(g_{\ell})
$$
and hence is in the $\BS$-module generated by $\bsu(R^{\perp})$.
Thus we have the equality $(\du(R)^{\perp}) = (\bsu(R^{\perp}))$ of $\BS$-modules. Therefore $$\du(\mathcal{P})^{!} =  \bsu(\mathcal{P}^{!}) ~\mbox{and}~ \du(\mathcal{P}) = \bsu(\mathcal{P}^{!})^{!}.$$

To prove the ``only if" part, suppose $R=0$. Then we have $\du(R)=0\subseteq \mathcal{T}(\du(V))$ and hence $\du(R)^\perp=\mathcal{T}(\bsu(V^\vee))(3)$. On the other hand, $R^\perp=\mathcal{T}(V^\vee)(3)$ which has a basis $e_i^\vee\circ_u e_j^\vee, 1\leq i, j\leq n, u\in \{\mathrm{I,II,III}\}$. Then a linear spanning set of $\bsu(\mathcal{T}(V^\vee)(3))$ is given by $\bsu_v(e_i^\vee\circ_u e_j^\vee), 1\leq i, j\leq n, v\in \{x, y, z\}, u\in \{\mathrm{I,II,III}\}$ in  Eq.~(\mref{eq:bsubase}). Thus the dimension of $\bsu(\mathcal{T}(V^\vee)(3))$ is at most $9n^2$, while the dimension of $\mathcal{T}(\bsu(V^\vee))(3)$ is
$$3 \dim(\bsu(V^\vee)^{\ot 2})=3 (2n)^2= 12 n^2.$$
Hence $\bsu(\mathcal{T}(V^\vee)(3))$ is a proper subspace of $\mathcal{T}(\bsu(V^\vee))(3)$ and thus
$\du(\mathcal{P})^{!} \neq  \bsu(\mathcal{P}^{!})$.
\end{proof}

\begin{theorem}\mlabel{thm:tdutsu}
Let $\bfk$ be an infinite field. Let $\mathcal{P} = {\mathcal T}(V)/(R)$ be a binary quadratic operad. Then
$$
\tdu(\mathcal{P})^{!} =  \tsu(\mathcal{P}^{!})
$$
if and only if $R\neq 0$.
\end{theorem}
\begin{proof}
The proof is similar to Theorem \mref{thm:dusu}.
\end{proof}

Taking $\mathcal{P}$ to be the operad of associative algebra in Theorem~\mref{thm:tdutsu}, we get the result of Loday and Ronco~\cite[theorem 3.1]{LR} that the triassociative algebra and the tridendriform algebra are in Koszul dual to each other.

More generally, Theorem~\mref{thm:dusu} and Theorem~\mref{thm:tdutsu} make it straightforward to compute the generating and relation spaces of the Koszul duals of the operads of some existing algebras. We give the following examples as illustrations.

\begin{enumerate}
\item
The operad $DualCTD$~\cite{Zi} is defined to be the Koszul dual of the operad $CTD$ of the commutative tridendriform algebra. Since the latter operad is $\tsu(\comm)$~\cite{BBGN}, we have
  $$DualCTD =\tsu(\comm)^!=\tdu(\comm^!)=\tdu(\lie),$$
which is precise is $\TriLeib$, the operad of the triLeibniz algebra in Proposition~\mref{prop:triLeib}. Thus we easily obtain the relations of $DualCTD$. See Eq.~(\ref{eq:trileib})
\item
The operad of the commutative quadri-algebra is the Kozul dual of $\bsu(\zinb)$ and hence is $\du(\leib)$. Thus its relations can be easily computed.
\item
The Kozul dual of $\bsu(\Prelie)$, the operad of the L-dendriform algebra, is $\du(\Perm)$ and hence can be easily computed.
\item
The operad $L$-$quad$~\cite{BLN} of the L-quadri-algebra is shown to be $\bsu(L\text{-}dend)=\bsu(\bsu(Lie))$ in~\cite{BBGN}. Thus the dual of $L\text{-}quad$ is $\du(\du(\Perm))$ and can be easily computed.
\end{enumerate}

\subsection{Replicators and Manin white products}

As a preparation for later discussions, we recall concepts and notations on Manin white product, most following~\mcite{Va}.

Ginzburg and Kapranov defined in \mcite{GK} a morphism of operads $\Phi: \mathcal{T}(V \otimes W) \rightarrowtail \mathcal{T}(V) \otimes \mathcal{T}(W)$. Let $\calp=\mathcal{T}(V)/(R)$ and $\calq=\mathcal{T}(W)/(S)$ be two binary quadratic operads with finite-dimensional generating spaces. Consider the composition of morphisms of operads
$$
\xymatrix{
\mathcal{T}(V \otimes W) \ar[r]^{\Phi}& \mathcal{T}(V) \otimes \mathcal{T}(W) \ar[rr]^{\pi_{\calp} \otimes \pi_{\calq}}&&\calp \otimes \calq,
}
$$
where $\pi_{\calp}: {\mathcal T}(V)\to {\mathcal P}$ and $\pi_{\calq}: {\mathcal T}(W)\to {\mathcal Q}$ are the natural projections.
Its kernel is $(\Phi^{-1}(R \otimes \mathcal{T}(W) + \mathcal{T}(V) \otimes S))$, the ideal generated by $\Phi^{-1}(R \otimes \mathcal{T}(W) + \mathcal{T}(V) \otimes S)$.

\begin{defn}(\mcite{GK,Va})
{\rm Let $\calp=\mathcal{T}(V)/(R)$ and $\calq=\mathcal{T}(W)/(S)$ be two binary quadratic operads with finite-dimensional generating spaces. The {\bf Manin white product} of $\mathcal P$ and $\mathcal Q$ is defined by
$$\calp \bigcirc \calq:=\mathcal{T}(V \otimes W )/(\Phi^{-1}(R \otimes \mathcal{T}(W) + \mathcal{T}(V) \otimes S)).$$}
\label{de:white}
\end{defn}

In general, the white Manin product difficult to compute  when the operads are given in terms of generators and relations.
Theorem~\mref{thm:whiteproduct} provides a convenient way to compute the white Manin product of a binary quadratic operad with the operad $\Perm$ or ${\it ComTrias}$ by relating them to the duplicator and triplicator.

\begin{theorem}\mlabel{thm:whiteproduct}
Let $\mathcal{P}=\mathcal{T}(V)/(R)$ be a binary quadratic operad with $R\neq 0$. We have the isomorphism of operads $$\du(\mathcal{P})\cong \Perm \bigcirc \mathcal{P}, \quad
\tdu(\mathcal{P})\cong {\it ComTrias} \bigcirc \mathcal{P}.
$$
\end{theorem}
\begin{proof}
By~\mcite{BBGN}, we have the isomorphisms of operads
$$\bsu(\mathcal{P}^{!}) \cong \Prelie \bullet  \mathcal{P}^{!}, \quad \tsu(\mathcal{P}^{!})\cong \postlie \bullet \mathcal{P}^{!}.$$
Since $\Prelie^{!} \cong \Perm$, $\postlie^{!} = {\it ComTrias}$ and $(\mathcal{P} \bullet \mathcal{Q})^{!} \cong \mathcal{P}^{!} \bigcirc \mathcal{Q}^{!}$, we obtain
$$
\du(\mathcal{P}) \cong (\bsu(P^{!}))^{!} \cong (\Prelie \bullet \mathcal{P}^{!})^{!} \cong \Perm \bigcirc \mathcal{P}.
$$
Similarly
$\tdu(\mathcal{P})\cong {\it ComTrias} \bigcirc \mathcal{P}.$
\end{proof}

By taking replicators of suitable operads $\calp$, we immediately get
\begin{coro}
\begin{enumerate}
\item
(\mcite{Va})\;$\Perm \bigcirc \lie = \leib$ and $\Perm \bigcirc \ass = \diass$.
\mlabel{it:leib}
\item
(\mcite{Uc})\; $\Perm \bigcirc {\it Pois}= {\it DualPrePois}$.
\mlabel{it:dualprepois}
\item
${\it ComTriass} \bigcirc \ass={\it Triass}$.\mlabel{it:Ass}
\end{enumerate}
\end{coro}

By a similar argument as for Theorem ~\mref{thm:whiteproduct} we obtain
\begin{prop}
Let $\calp=\mathcal{T}_{ns}(V)/(R)$ be a binary quadratic nonsymmetric operad with $R\neq 0$. There is an isomorphism of nonsymmetric operads
$$ \du(\calp) \cong {\it Dias } \ \square  \ \calp \ , \quad \tdu(\calp) \cong {\it Trias} \ \square \  \calp \ ,$$
where $\square$ denotes the white square product~\mcite{Va} while $Dias$ and $Trias$ denote the nonsymmetric operads for the diassociative and triassociative algebras.
\end{prop}

\section{Replicators and Average operators on operads}
\mlabel{sec:rb}
In this section we establish the relationship between the duplicator and triplicator of an operad on one hand and the actions of the di-average and tri-average operators on the operad on the other hand. We will work with symmetric operads, but all the results also hold for nonsymmetric operads.

\subsection{Duplicators and di-average operators}
Averaging operators have been studied for associative algebras since 1960 by Rota and for other algebraic structures more recently~\mcite{Ag2,Cao,R1,Uc}.

\begin{defn}
{\rm
Let $(A,\cdot)$ be a $\bfk$-module $A$ with a binary operation $\cdot$.
\begin{enumerate}
\item
A {\bf di-average operator} on $A$ is a $\bfk$-linear map $P: A \longrightarrow A$ such that
\begin{eqnarray}
P(x \cdot P(y)) &=& P(x) \cdot P(y) = P(P(x) \cdot y), \quad \text{ for all } x, y\in A.\mlabel{eq:diav}
\end{eqnarray}
\item
Let $\lambda \in \bfk$. A tri-average operator of weight $\lambda$ on $A$ is a $\bfk$-linear map $P: A \longrightarrow A$ such that Eq.~(\mref{eq:diav}) holds and
\begin{eqnarray}
P(x)\cdot P(y) &=&\lambda P(xy),
\quad \text{ for all } x, y\in A.\mlabel{eq:triav} \end{eqnarray}
\end{enumerate}
}
\end{defn}
We note that a tri-average operator of weight zero is not a di-average operator. So we cannot give a uniform definition of the average operators as in the case of  Rota-Baxter algebras of weight $\lambda$.

We next consider the operation of average operators on the level of operads.
\begin{defn}\mlabel{defn:av}
{\rm Let $\gensp=\gensp(2)$ be an $\BS$-module concentrated in arity 2.
\begin{enumerate}
\item
Let $\bvp$ denote the $\BS$-module concentrated in arity 1 and arity 2 with $\bvp(2)=V$ and $\bvp(1)=\bfk\,P$, where $P$ is a symbol. Let $\mathcal{T}(\bvp)$ be the free operad generated by binary operations $\gensp$ and an unary operation $P\neq \id$.
\item
Define $\du(\gensp)=\gensp \otimes (\bfk \dashv \oplus \,\bfk \vdash)$ as in Eq.~(\mref{eq:tsp}), regarded as an $\BS$-module concentrated in arity 2. Define a linear map of graded vector spaces from $\du(V)$ to $\bvp$ by the following correspondence:
$$\xi:\quad \svec{\gop}{\dashv}\mapsto \gop\circ({\rm id}\otimes P),\quad \svec{\gop}{\vdash}\mapsto \gop\circ (P\otimes {\rm id}), \text{ for all } \gop\in \gensp,$$
where $\circ$ is the operadic composition. By the universality of the free operad, $\xi$ induces a homomorphism of operads that we still denote by $\xi$:
$$\xi:\mathcal{T}(\su(V))\to \mathcal{T}(\bvp).$$
\item
Let $\calp=\mathcal{T}(\gensp)/(R_\calp)$ be a binary operad defined by generating operations $\gensp$ and
relations $R_\calp$.
Let
$$\da_{\mathcal{P}}:=\{\gop\circ(P\otimes P)-P\circ\gop\circ(P\otimes {\rm id}), \gop\circ(P\otimes P)- P\circ\gop\circ({\rm id}\otimes P)\ |\ \gop\in \gensp\}.$$
Define the {\bf operad of di-average $\calp$-algebras}
by
$$\da(\calp):=\mathcal{T}(\bvp)/( {\mathrm R}_\calp,{\mathrm DA}_{\calp}).$$
Let $p_1:\mathcal{T}(\bvp)\to
\da(\calp)$ denote the operadic projection.
\end{enumerate}
}
\end{defn}

\begin{theorem}
\begin{enumerate}
\item Let $\calp$ be a binary operad. There is a morphism of operads
$$
\du(\calp) \longrightarrow \da(\mathcal{P}),
$$
which extends the map $\xi$ given in  Definition~\ref{defn:av}.
\item Let $A$ be a $\mathcal{P}$-algebra. Let $P:A\to A$ be a di-average operator. Then the following operations make $A$ into a
$\du(\opd)$-algebra:
$$x\dashv_jy:=x\circ_jP(y),\quad x\vdash_jy:=P(x)\circ_jy,\quad \forall \circ_j\in\mathcal{P}(2), \text{ for all } x,y\in A.$$
\end{enumerate}
\mlabel{thm:biav}
\end{theorem}
The proof is parallel to the case of triplicators in Theorem~\mref{thm:triav} for which we will prove in full detail.

When we take $\mathcal{P}$ be the operad of the associative algebra, Lie algebra or Poisson algebra,
we obtain the following results of Aguiar \mcite{Ag2}.

\begin{coro}
\begin{enumerate}
\item Let $(A,\cdot)$ be an associative algebra and $P:A \longrightarrow A$ be a di-averaging operator. Define two new operations on $A$ by
    $
    x \vdash y = P(x)\cdot y ~\mbox{and}~x \dashv y = x \cdot P(y).
    $
    Then $(A,\vdash,\dashv)$ is an associative dialgebra.
\item Let $(A,[,])$ be a Lie algebra and $P:A \longrightarrow A$ be a di-averaging operator. Define a new operation on $A$ by $
    \{x,y\} = [P(x),y].$
Then $(A,\{,\})$ is a left Leibniz algebra.
\item Let $(A,\cdot,[,])$ be a Poisson algebra and let $P:A\to A$ be a di-averaging operator. Define two new products on $A$ by
$x\circ y:=P(x)\cdot y,~\mbox{and}~ \{x,y\}:=[P(x),y].$
Then $(A,\circ,\{,\})$ is a dual left prePoisson algebra.
\end{enumerate}
\end{coro}

Combining Theorem~\mref{thm:biav} with Theorem~\mref{thm:whiteproduct}, we obtain the following relation between the Manin white product and the action of the di-average operator. It can be regarded as the interpretation of \cite[Theorem 3.2]{Uc} at the level of operads.
\begin{prop}\mlabel{diavprop}
For any binary quadratic operad $\mathcal{P} = \mathcal{T}(V)/(R)$, there is a morphism of operads
$$
\Perm \bigcirc \mathcal{P} \longrightarrow \da(\calp),
$$
defined by the following map
\begin{eqnarray*}
\Perm(2) \bigcirc \mathcal{P}(2) &\longrightarrow& \da(\calp),\\
\mu \otimes \omega &\longmapsto& \omega \circ (id \otimes P),\\
\mu' \otimes \omega &\longmapsto& \omega \circ (P \otimes id), \quad \omega \in \mathcal{P}(2),
\end{eqnarray*}
where $\mu$ denotes the generating operation of the operad $\Perm$.
\end{prop}

\subsection{Triplicators and tri-average operators}
\mlabel{ss:avtd}
In this section, we establish the relationship between the triplicator of an operad and the action of the tri-average operator with a nonzero weight on the operad. For simplicity, we assume that the weight of the tri-average operator is one.

\begin{defn}\mlabel{ta}
{\rm Let $\gensp=\gensp(2), \bvp$ and $\mathcal{T}(\bvp)$ as defined in Definition~\mref{defn:av}.
\begin{enumerate}
\item
Let $\tdu(\gensp)=\gensp \otimes (\bfk \dashv \oplus\,\bfk \vdash\oplus\,\bfk \perp)$ in Eq.~(\mref{eq:ttsp}), seen as an $\BS$-module concentrated in arity 2. Define a linear map of graded vector spaces from $\tdu(V)$ to $\bvp$ by the correspondence
$$
\eta:\quad \svec{\gop}{\dashv}\mapsto \gop\circ({\rm id}\otimes P),\quad \svec{\gop}{\vdash}\mapsto \gop\circ (P\otimes {\rm id}),\quad \svec{\gop}{\cdot}\mapsto \gop,
$$
where $\circ$ is the operadic composition. By the universality of the free operad, $\eta$ induces a homomorphism of operads:
$$
\eta :\mathcal{T}(\tdu(V))\to \mathcal{T}(\bvp).
$$
\item
Let $\calp=\mathcal{T}(\gensp)/(R_\calp)$ be a binary operad defined by generating operations $\gensp$ and
relations $R_\calp$.
Let
\begin{eqnarray*}
{\mathrm TA}_{\mathcal{P}}&:=&\{\gop\circ(P\otimes P)-P\circ\gop\circ(P\otimes {\rm id}), \gop\circ(P\otimes P)- P\circ\gop\circ({\rm id}\otimes P),\\
&&\gop\circ(P\otimes P)- P\circ\gop~|~ \gop\in \gensp\}.
\end{eqnarray*}
Define the {\bf operad of tri-average $\calp$-algebras of weight one} by
$${\mathrm TA}(\calp):=\mathcal{T}(\bvp)/({\mathrm R}_\calp,{\mathrm TA}_{\calp}).$$
Let $p_1:\mathcal{T}(\bvp)\to {\mathrm TA}(\calp)$
denote the operadic projection.
\end{enumerate}
}
\end{defn}

We first prove a lemma relating triplicators and tri-average operators.

\begin{lemma}
Let $\calp=\mathcal{T}(\gensp)/(R_\calp)$ be a binary operad and
let $\tau\in \calt(\gensp)$ with $\lin(\tau)$.
\begin{enumerate}
\item
For each $\bar{\tau} \in \tdu(\tau)$, we have
\begin{equation}
P\circ \eta(~\bar{\tau}~) \equiv \tau\circ P^{\ot n} \mod ({\mathrm R}_\calp, {\mathrm TA}_\calp).
\mlabel{eq:tpxi}
\end{equation}
\mlabel{it:tpxi}
\item
For $\varnothing \neq J\subseteq \lin(\tau)$, let $P^{\ot n,J}$ denote the $n$-th tensor power of $P$ but with the component from $J$ replaced by the identity map. So, for example, for the two inputs $x_1$ and $x_2$ of $P^{\ot 2}$, we have $P^{\ot 2, \{x_1\}}=P\ot {\rm id}$ and $P^{\ot 2, \{x_1,x_2\}}={\rm id}\ot {\rm id}$. Then for each $\bar{\tau}_{J} \in \tdu_{J}(\tau)$, we have
\begin{equation}
\eta(\bar{\tau}_{J}) \equiv \tau \circ (P^{\ot n, J}) \mod ({\mathrm R}_\calp, {\mathrm TA}_\calp)\,.
\mlabel{eq:txidu}
\end{equation}
\label{it:txidu}
\end{enumerate}
\label{lem:txidu}
\end{lemma}

\begin{proof}
(\mref{it:tpxi}).
We prove by induction on $|\lin(\tau)|\geq 1$. When $|\lin(\tau)|=1$, $\tau$ is the tree with one leaf standing for the identity map. Then we have $\eta(~ \tdu(\tau)~)=\tau$, $P\circ \eta(~\tdu(\tau)~)=P=\tau\circ P$. Assume the claim has been proved for $\tau$ with $|\lin(\tau)|=k$ and consider a $\tau$ with $|\lin(\tau)|=k+1$. Then from the decomposition $\tau=\tau_\ell\vee_{\gop} \tau_r$, we have
$\tdu(\tau)=\tdu(\tau_\ell) \vee_{\svec{\gop}{\dagger}} \tdu(\tau_r)$. Recall that $\tdu(\tau)$ is a set of labeled trees. For each $\bar{\tau} \in \tdu(\tau)$, there exist $\bar{\tau}_{\ell} \in \tdu(\tau)_{\ell}$ and $\bar{\tau}_{r} \in \tdu(\tau_{r})$ such that
$$
\bar{\tau} \in \left\{\bar{\tau}_{\ell} \vee_{\svec{\gop}{\vdash}} \bar{\tau}_{r},\bar{\tau}_{\ell} \vee_{\svec{\gop}{\dashv}} \bar{\tau}_{r},\bar{\tau}_{\ell} \vee_{\svec{\gop}{\cdot}} \bar{\tau}_{r} \right\}.
$$
If $\bar{\tau} = \bar{\tau}_{\ell} \vee_{\svec{\gop}{\vdash}} \bar{\tau}_{r}$, then we have
\begin{eqnarray*}
P\circ \eta(\bar{\tau}) &=& P\circ \eta(\bar{\tau}_{\ell} \vee_{\svec{\gop}{\vdash}} \bar{\tau}_{r}) \\
&=& P \circ \omega \circ ((P \circ \eta (\bar{\tau}_{\ell})) \otimes \eta(\bar{\tau}_{r}))\\
&\equiv& \omega \circ ((P \circ \eta (\bar{\tau}_{\ell})) \otimes (P \circ \eta (\bar{\tau}_{r}))) \mod ({\mathrm R}_\calp, {\mathrm TA}_\calp) \\
&=& \omega \circ ((\bar{\tau}_{\ell} \circ P^{\otimes |\lin(\tau_{\ell})|})  \otimes (\bar{\tau}_{r} \circ P^{\otimes |\lin(\tau_{r})|})) \quad \text{(by induction hypothesis)}\\
&=& \omega \circ (\bar{\tau}_{\ell} \otimes \bar{\tau}_{r}) \circ P^{\otimes (k+1)}\\
&=&(\bar{\tau}_{\ell} \vee_{\svec{\gop}{\vdash}} \bar{\tau}_{r}) \circ P^{\otimes (k+1)}\\
&=& \bar{\tau} \circ P^{\otimes (k+1)}.
\end{eqnarray*}
Similarly, we have
\begin{eqnarray*}
P\circ \eta(~\bar{\tau}_{\ell} \vee_{\svec{\gop}{\dashv}} \bar{\tau}_{r}~) \equiv \bar{\tau} \circ P^{\otimes (k+1)} \mod ({\mathrm R}_\calp, {\mathrm TA}_\calp),\\
P\circ \eta(~\bar{\tau}_{\ell} \vee_{\svec{\gop}{\cdot}} \bar{\tau}_{r}~) \equiv \bar{\tau} \circ P^{\otimes (k+1)} \mod ({\mathrm R}_\calp, {\mathrm TA}_\calp).
\end{eqnarray*}

\noindent
(\mref{it:txidu}). We again prove by induction on $|\lin(\tau)|$. When $|\lin(\tau)|=1$, then $x$ is the only leaf label of $\tau$ and $|\tdu_{x}(\tau)| = 1$. Thus we have
$$ \eta(\bar{\tau}_{x})=\eta(x)=x = \tau \circ (P^{\ot 1, x}).$$
Assume that the claim has been proved for all $\tau$ with $|\lin(\tau)|=k$ and consider $\tau$ with $|\lin(\tau)|=k+1$. Write $\tau=\tau_\ell \vee_\gop \tau_r$. Let $J$ be a nonempty subset of $\lin(\tau)$. If $J\subseteq  \lin(\tau_\ell)$, then by the definition of $\tdu_J(\tau)$, for each $\bar{\tau}_{J} \in \tdu_{J}(\tau)$, there exist $\bar{\tau}_{J,\ell} \in \tdu_{J}{\tau}_{\ell}$ and $\bar{\tau}_{J,r} \in \tdu_{\varnothing}{\tau}_{r}$ such that $\bar{\tau}_{J} = \bar{\tau}_{J,\ell}  \vee_{\svec{\gop}{\dashv}}   \bar{\tau}_{J,r}$. Then we have
{\allowdisplaybreaks
\begin{eqnarray*}
\eta(\bar{\tau}_{J})&=& \eta(\bar{\tau}_{J,\ell}  \vee_{\svec{\gop}{\dashv}}   \bar{\tau}_{J,r})\\
&=& \gop \circ (\eta(\bar{\tau}_{J,\ell} \ot P\circ \eta(\bar{\tau}_{J,r}))\\
&\equiv & \gop \circ \left ( (\tau_\ell \circ P^{\ot |\lin(\tau_\ell)|,J})\ot (\tau_r\circ P^{\ot |\lin(\tau_r)|})\right) \mod ({\mathrm R}_\calp,{\mathrm TA}_\calp)\\
 && \quad \quad \quad \text{(by induction hypothesis and Item~(\mref{it:tpxi}))} \\
&=& \tau \circ P^{\ot (k+1),J}.
\end{eqnarray*}
}
When $J\subseteq \lin(\tau_r)$, the proof is the same.
When $J\not\subseteq \lin(\tau_\ell)$ and $J\not\subseteq \lin(\tau_r)$, for each $\bar{\tau}_{J} \in \tdu_{J}(\tau)$, there exist $\bar{\tau}_{J,\ell} \in \tdu_{J\cap\lin(\tau_\ell)}{\tau}_{\ell}$ and $\bar{\tau}_{J,r} \in \tdu_{J\cap\lin(\tau_r)}{\tau}_{r}$ such that $\bar{\tau}_{J} = \bar{\tau}_{J,\ell}  \vee_{\svec{\gop}{\cdot}}  \bar{\tau}_{J,r}$. Then by the same argument we have
{\allowdisplaybreaks
\begin{eqnarray*}
\eta(\bar{\tau}_{J})
&\equiv & \gop \circ \left ( (\tau_\ell \circ P^{\ot |\lin(\tau_\ell)|,J\cap\lin(\tau_\ell)})\ot (\tau_r\circ P^{\ot |\lin(\tau_r)|,J\cap\lin(\tau_r)})\right) \mod ({\mathrm R}_\calp, {\mathrm TA}_\calp) \\
&=& \tau \circ P^{\ot (k+1),J}.
\end{eqnarray*}
}
This completes the induction.

\end{proof}

\begin{theorem}\mlabel{thm:triav}
Let $\calp$ be a binary operad.
\begin{enumerate}
\item There is a morphism of operads
$$
\tdu(\mathcal{P}) \longrightarrow {\mathrm TA}(\calp),
$$
which extends the map $\eta$ given in Definition \mref{ta}.
\item Let $A$ be a $\calp$-algebra. Let $P: A \longrightarrow A$ be a tri-average operator of weight one. Then the following operations make $A$ into a $\tdu(\mathcal{P})$-algebra:
    $$
    x \dashv_{j} y = x \circ_{j} P(y), \quad x \vdash_{j} y = P(x) \circ_{j} y, \quad x \cdot_{j} y = x \circ_{j} y, \quad \text{for all } \circ_{j} \in \calp(2).
    $$
    \mlabel{it:triavb}
\end{enumerate}
\end{theorem}

\begin{proof}
The second statement is just the interpretation of the first statement on the level of algebras. So we just need to prove the first statement. Let $R_{\tdu(\calp)}$ be the relation space of $\tdu(\calp)$.
By definition, the relations of $\tdu(\calp)$ are generated by $\tdu_J(r)$ for locally homogeneous $r=\sum_i c_i \tau_i \in R_\calp$, where $\varnothing\neq J\subseteq \lin (\tau_i)$.
By Eqs.(\mref{eq:tpxi}) and (\mref{eq:txidu}), we have
$$
\eta \left (\sum_{i} c_i \bar{(\tau_{i})}_{J}\right) = \sum_i c_i \eta(\bar{(\tau_{i})}_{J}) \equiv \sum_i c_i \tau_i \circ P^{\ot n,J}
 \equiv \left (\sum_i c_i \tau_i \right) \circ P^{\ot n, J} \mod ({\mathrm R}_\calp, {\mathrm TA}_\calp).
$$
Hence $\eta(R_{\tdu(\calp)}) \subseteq ({\mathrm R}_\calp, {\mathrm TA}_\calp)$ and $\eta$ induces a morphism of operads
$$
\bar{\eta}: \tdu(\calp) \longrightarrow {\mathrm TA}(\calp).
$$
This proves the first statement.
\end{proof}

\begin{coro}
\begin{enumerate}
\item
Let $A$ be an associative algebra and let $P: A \longrightarrow A$ be a tri-average operator on $A$.  Then the new operations defined in Theorem~\mref{thm:triav}(\ref{it:triavb}) makes it into an associative trialgebra.
\item
Let $L$ be a Lie algebra and let $P: L \longrightarrow L$ be a tri-average operator on $L$. Then the operations defined in Theorem~\mref{thm:triav}(\ref{it:triavb}) make it into a triLeibniz algebra.
\end{enumerate}
\end{coro}

\smallskip

\noindent {\bf Acknowledgements: } C. Bai would like to thank the support by NSFC (10920161, 11271202).
L. Guo acknowledges support from NSF grant DMS 1001855. The authors thank Kolesnikov for informing them the related papers~\mcite{GK2,Ko,KV}.

\end{document}